\documentclass[11pt,reqno]{amsart}
\usepackage[top=3.5cm,bottom=3.5cm,left=2.5cm,right=2.5cm,heightrounded,bindingoffset=0mm]{geometry}

\usepackage[T1]{fontenc}
\usepackage[utf8]{inputenc}
\usepackage[english]{babel}
\usepackage{microtype}
\usepackage{lmodern}
\usepackage{amsmath,amsthm,amssymb,mathrsfs}

\usepackage{color}
\usepackage{graphicx}
\usepackage{tikz}
\newcommand{\bq}{\begin{equation}}
\newcommand{\eq}{\end{equation}}
\newcommand{\bqa}{\begin{eqnarray*}}
\newcommand{\eqa}{\end{eqnarray*}}

\theoremstyle{plain}
\newtheorem{theo}{Theorem}[section]
\newtheorem{prop}[theo]{Proposition}
\newtheorem{lemm}[theo]{Lemma}
\newtheorem{coro}[theo]{Corollary}

\newtheorem{defi}[theo]{Definition}
\newtheorem{rema}[theo]{Remark} 
\newtheorem{nota}[theo]{Notation}
\DeclareMathOperator{\di}{div}

\DeclareSymbolFont{pletters}{OT1}{cmr}{m}{sl}
\DeclareMathSymbol{s}{\mathalpha}{pletters}{`s}

\def\eps{\varepsilon}
\def\na{\nabla}

\def\mez{\frac{1}{2}}

\def\Rr{\mathbb{R}}
\def\Nn{\mathbb{N}}

\def\cG{\mathcal{G}}

\def\cF{\mathcal{F}}

\def\cL{\mathcal{L}}
\def\L1{\mathcal{L}^{(1)}}
\def\L2{\mathcal{L}^{(2)}}
\def\L3{\mathcal{L}^{(3)}}

\def\p{\partial}

\def\na{\nabla}

\def\T{\mathbb{T}}

\def\wt{\widetilde}

\def\fd{\mathfrak{d}}

\def\rH{\mathring{H}}
\def\rC{\mathring{C}}

\numberwithin{equation}{section}
\def\om{\omega}

\def\a{\alpha}
\title[Slowly traveling gravity waves for Darcy flow]{Slowly traveling gravity waves  for Darcy flow:\\
existence and stability of large waves}

\date{\today}

\author{John Brownfield} 
\address{
Department of Mathematics\\
University of Maryland\\
College Park, MD 20742, USA
}
\email[J. Brownfield] 
{brownfield.carrington@gmail.com}

\author{Huy Q. Nguyen}
\address{
Department of Mathematics\\
University of Maryland\\
College Park, MD 20742, USA
}
\email[H. Nguyen]{hnguye90@umd.edu}

\usepackage{hyperref}
\hypersetup{colorlinks={true},linkcolor={blue},citecolor=blue}

\begin{document}
\newcommand{\huy}[1]{{\color{orange} \textbf{H:} #1}}
\maketitle
\begin{abstract}
We study surface gravity  waves for viscous  fluid flows governed by Darcy's law. The free boundary is acted upon by an external
pressure posited to be in traveling wave form with a periodic profile. It has been proven that for any given speed,  small external pressures  generate small periodic traveling waves that are asymptotically stable. In this work, we construct a class of slowly traveling waves that are of arbitrary size and asymptotically stable. Our results are valid in all dimensions and for both the finite and infinite depth cases.
\end{abstract}

%\keywords{Darcy law, Muskat problem, Hele-Shaw, traveling waves, capillarity}
%
%\noindent\thanks{\em{ MSC Classification:  .}}

\section{Introduction}
In 1847 Stokes \cite{Stokes} proposed a formal construction of small-amplitude periodic traveling water waves. This pioneering work has lead to the development of an important area in fluid mechanics: traveling surface waves for {\it inviscid} fluids. Spectacular developments have been achieved, from rigorous constructions of small-amplitude waves to that of large-amplitude and extreme waves,  from formal to rigorous stability analysis of traveling waves of various types, to name a few. We refer to \cite{Ebbandflow} for a recent survey on some directions in this vast subject.

The present paper is concerned with traveling surface waves for {\it viscous} fluids, which is much less developed due to the obvious obstruction that energy is dissipated in viscous flows. Nevertheless, experiments \cite{DCDA1, DCDA2, MasnadiDuncan, ParkCho1} have revealed the existence of traveling waves for viscous fluids when they are appropriately {\it forced}. Motivated by these experimental findings,   Leoni-Tice \cite{LeoniTice} constructed for the first time {\it small} traveling waves for the free boundary gravity and capillary-gravity incompressible Navier-Stokes equations forced by a {\it small} bulk force and a {\it small} external stress tensor on the free boundary. The  traveling waves in \cite{LeoniTice} are asymptotically flat at spatial infinity. The construction in \cite{LeoniTice} has been extended in a variety of directions, including periodic and tilted configurations \cite{KoganemaruTice1}, multi-layer configurations \cite{StevensonTice1}, the vanishing wave speed limit \cite{StevensonTice2}, the Navier-slip boundary conditions \cite{KoganemaruTice1}, the compressible Navier-Stokes equations \cite{StevensonTice3}, and the shallow water equations \cite{StevensonTice3}. For Darcy flows, including flows in porous media and vertical Hele-Shaw cells, a similar construction of small traveling waves was obtained in  \cite{NguyenTice}, and it was proven therein that small periodic traveling waves generated by  small external pressures on the free boundary are {\it asymptotically stable}. 

All of the aforementioned viscous free boundary problems admit  the flat free boundary as a trivial solution when no external forces are applied. Small traveling waves corresponding to small external forces were constructed perturbatively via some Implicit Function Theorem argument, including Nash-Moser ones \cite{StevensonTice3, StevensonTice4}. In parallel to the inviscid theory, the following problems remain to be investigated for the viscous theory: 

(I) Construction of large traveling waves  from large external forces. 

(II) Stability of traveling waves.  

In \cite{Nguyen2023-capillary}, by means of a global continuation theory, {\it large periodic} traveling waves for Darcy flows with surface tension were constructed from large external pressures on the free boundary. The constructed waves assume  any specified speed. The stability of large waves in \cite{Nguyen2023-capillary} is yet to be studied, while small waves have been proven to be asymptotically stable \cite{NguyenTice}. The purpose of the present work is to prove the existence of a class of periodic {\it slowly} traveling waves that are simultaneously of {\it arbitrary size} and {\it asymptotically stable}. Our results are valid in all dimensions and for both the finite and infinite depth cases. In order to discuss the main results and ideas, we next describe the setup of the problem.  
    \subsection{Surface waves for Darcy flow with externally applied pressure} 
    We consider surface waves for flows modeled by Darcy's law 
    \bq\label{Darcy}
    \mu u +\na_{x, y}p=-\rho ge_y,\quad \di u=0\quad\text{in }\Omega_{\eta}.
    \eq
    The fluid domain 
    \bq\label{def:Omega}
    \Omega_\eta:=\{(x, y)\in \Rr^d\times \Rr: -b<y<\eta(x, t)\}
    \eq
    lies below the free boundary 
 \bq\label{def:Sigma}
\Sigma_\eta=\{(x, \eta(x, t)): x\in \Rr^d\}
 \eq
 which is the graph of the unknown function $\eta(x, t): \Rr^d\times [0, \infty)\to \Rr$. Here, $\{y=-b\}$, $b>0$, is the fixed bottom. We  also consider  the infinite depth case, i.e. deep fluids,  
 \bq\label{def:Omega:id}
    \Omega_\eta=\{(x, y)\in \Rr^d\times \Rr: y<\eta(x, t)\}.
    \eq
The free boundary evolves according to the kinematic boundary condition 
\bq
\p_t\eta=u\cdot N\vert_{\Sigma_\eta},\quad N:=(-\na_x\eta, 1).
\eq
The fluid is acted upon by an external pressure posited to be in traveling wave form with speed $\gamma$:
\bq\label{epressure}
\psi(x, y, t)=\phi(x- \gamma e_1t),
\eq
 where we have assumed without loss of generality that $e_1$ is the direction of propagation. With \eqref{epressure} we  consider external pressures that are uniform in the vertical direction.

Following \cite{NgPa}, it was shown in \cite{NguyenTice}  that the free boundary problem \eqref{Darcy}-\eqref{epressure} can be reformulated in terms of the Dirichlet-Neumann operator as
\bq\label{eq:eta}
\p_t \eta=-\frac{\rho g}{\mu}G[\eta](\eta+\psi).
\eq
 The Dirichlet-Neumann operator $G[\eta]$ associated to the domain $\Omega_\eta$ is recalled in Definition \ref{def:DN}.  Upon rescaling time we assume henceforth that 
 \[
 \frac{\rho g}{\mu}=1.
 \]  
 In the moving frame of the external pressure \eqref{epressure}, we have $\eta(x,t) = \widetilde\eta(x-\gamma e_1 t, t)$, where  $\widetilde\eta$ satisfies
    \begin{equation} \label{eq:eta}
            \partial_t \widetilde\eta - \gamma \partial_1 \widetilde \eta = - G[\widetilde\eta](\widetilde\eta + \phi).
    \end{equation}
    In what follows, we will drop the tilde over $\eta$ for convenience. A traveling wave  is a time-independent solution $\eta$ of \eqref{eq:eta}, i.e. $\eta$ satisfies 
      \bq\label{eq:tw}
              -\gamma \partial_1 \eta = - G[\eta](\eta + \phi).
              \eq
\subsection{Main results} 
Our first main result - Theorem \ref{theo:existence} below - concerns the existence of {\it large} traveling waves. We observe that for any given $\phi$, 
\bq\label{origin}
(\eta, \gamma)=(-\phi, 0)
\eq
is a solution of \eqref{eq:tw}, provided $-\phi>-b$ in the finite depth. Our idea is to perturb around the {\it large} solution \eqref{origin} to obtain slowly traveling waves. Since $-\phi$ can be arbitrarily large, so can $\eta$; moreover, $\eta$ stays above the bottom in the finite depth case since $-\phi$ does so. \begin{theo}\label{theo:existence}
 Let $k \geq 1$, $\alpha \in (0,1)$, and $\phi \in \rC^{k,\alpha}(\mathbb{T}^d)$. In the finite depth case we assume that $-\phi>-b$. There exists a small number $\delta_0 = \delta_0(\|\phi\|_{C^{k,\alpha}})>0$ such that the following holds. For all $\delta \in (0, \delta_0)$, there exists $\eps = \eps(\delta, \|\phi\|_{C^{k,\alpha}}) > 0$ such that if $|\gamma| < \eps$ then \eqref{eq:tw} has a unique solution $\eta \in \rC^{k, \alpha}(\mathbb{T}^d)$ satisfying $\|\eta+\phi\|_{C^{k, \alpha}} \le  \delta$. Moreover, for any $\delta\in (0, \delta_0)$, the mapping
     \[
         (-\eps, \eps)\ni \gamma \mapsto \eta\in \overline{B}_\delta(-\phi)\subset \rC^{k, \a}(\T^d)
         \]
         is Lipschitz continuous. 
    \end{theo}
In the  second main result - Theorem \ref{theo:stability:intro} below - we identify a class of {\it large} traveling waves that are {\it asymptotically stable} in Sobolev spaces.    \begin{theo}\label{theo:stability:intro}
        Let $\Nn\ni s >1 +\frac{d}{2}$ and $\phi \in H^{s+\mez}(\mathbb{T}^d)$. Assume that $\eta_* \in W^{s+2, \infty} (\mathbb{T}^d)$ is a traveling wave  with speed $\gamma \in \mathbb{R}$.  There exists a nonincreasing function (in each argument) $\om: \Rr_+\times \Rr_+\to \Rr_+$ such that if  
        \bq\label{cd:stability}
       \|\eta_* + \phi\|_{H^{s+\frac12}}<\om\big(\| \eta_*\|_{W^{s+2, \infty}}, \| \phi\|_{H^s}\big)
       \eq
        then $\eta_*$ is asymptotically stable with respect to perturbations in $\rH^s(\T^d)$.
    \end{theo}
    We refer to Theorem \ref{theo:stability}  for a more precise version of Theorem \ref{theo:stability:intro}. We note that the condition \eqref{cd:stability} imposes a restriction on the size of $\eta_*+\phi$ but not on $\eta_*$, thereby allowing for large $\eta_*$. Theorem \ref{theo:existence} provides an example of a class of large traveling waves satisfying condition \eqref{cd:stability}. Indeed, according to Theorem \ref{theo:existence}, for any given large profile $\phi\in \rC^{s+2, \a}(\T^d)$ of the external pressure, there exists a slowly traveling wave $\eta_*$ provided 
    \[
    \|\eta_*+\phi\|_{C^{s+2, \alpha}}\le \delta<\delta_0=\delta_0(\| \phi\|_{C^{s+2, \a}}).
    \]
     But then $\| \eta_*\|_{W^{s+2, \infty}}< \delta_0+\| \phi\|_{W^{s+2, \infty}}$, and thus \eqref{cd:stability} is satisfied if  
     \[
     \delta<\om(\delta_0+\|\phi \|_{W^{s+2, \infty}}, \| \phi\|_{H^s})
     \]
     which depends only on the given $\phi$. We thus obtain the following corollary on the asymptotic stability of slowly traveling waves.
        \begin{coro}\label{coro}
        Let $\Nn\ni s >1 +\frac{d}{2}$, $\alpha\in (0, 1)$, and $\phi\in \rC^{s+2, \alpha}(\T^d)$. In the finite depth case we assume that $-\phi>-b$. Let $\delta_0 = \delta_0\big(\|\phi\|_{C^{s+2,\alpha}}\big)$ be the constant given in Theorem \ref{theo:existence}, and  let $\om: \Rr_+\times \Rr_+\to \Rr_+$ be the nonincreasing function given in Theorem \ref{theo:stability:intro}. Then any  slowly traveling wave  $\eta$ constructed in Theorem \ref{theo:existence} with  
        \bq
        \| \eta+\phi\|_{C^{s+2, \a}}<\min\Big(\delta_0, \om\big(\delta_0+\|\phi \|_{W^{s+2, \infty}},  \| \phi\|_{H^s}\big)\Big)
        \eq
          is  asymptotically stable with respect to perturbations in  $\rH^s(\T^d)$.
        \end{coro}
%Dealing with large solutions, the proofs of Theorems \ref{theo:existence} and \ref{theo:stability:intro} are based mostly upon physical space arguments, as opposed to Fourier space arguments in the previous works \cite{LeoniTice,  NguyenTice, StevensonTice2}.
       
       Corollary \ref{coro} demonstrates a stark contrast between periodic traveling surface waves for inviscid fluids and those for viscous fluids. For the inviscid  water problem, even very small periodic traveling waves (i.e. Stokes waves) are known to be unstable  in Sobolev spaces \cite{BenjaminFeir, BM, NguyenStrauss, ChenSu}.  
       
       The proofs of Theorem \ref{theo:existence} and Theorem \ref{theo:stability:intro} exploit various properties of the Dirichlet-Neumann operator for far-from-flat domains. These include invertibility, contraction, coercivity, and commutator with partial derivatives.  In Section \ref{sec:DN}, we first recall some known results on the Dirichlet-Neumann operator in H\"older and Sobolev spaces, then we prove commutator estimates for the Dirichlet-Neumann operator with partial derivatives. Our commutator estimates are sharp in the sense that they exhibit a full gain  of one derivative; this is a result of independent interest. Section \ref{sec:existence} is devoted to the proof of Theorem \ref{theo:existence}. In Section \ref{sec:lin}, we prove well-posedness of the linear dynamics around large traveling waves. Finally, we prove Theorem \ref{theo:stability:intro} in Section \ref{sec:stability}.
       \begin{nota} We fix the following notation throughout this paper:
       \begin{itemize}
    \item $\mathbb{N}=\{0, 1, 2. \dots\}$.
    \item For $a>0$, $\lceil a\rceil$ denotes the smallest integer greater than or equal to $a$.
    \item $B_r(a)$ denotes the open ball of radius $r$ centered at $a$. $\overline{B}_r(a)$ denotes the closure of $B_r(a)$.
       \item If $X$ is a space of integrable  functions on $\T^d$, we set
       \[
\mathring{X}=\{u\in X: \int_{\T^d}u=0\}.
       \]
       \item $[A, B]:=AB-BA$ is the commutator of the operators $A$ and $B$.
       \end{itemize}
       \end{nota}
\section{The Dirichlet-Neumann operator}\label{sec:DN}
 Throughout this section we assume that in the finite depth case \eqref{def:Omega} the free boundary is separated from the bottom by a positive distance, i.e.
        \bq\label{cd:fd}
        \inf_{x\in \T^d}(\eta(x)+b)\ge \fd>0.
        \eq
    \subsection{Known results}
    \begin{defi}\label{def:DN} Let $\Omega_\eta$ be as in \eqref{def:Omega} (finite depth) or \eqref{def:Omega:id} (infinite depth). The Dirichlet-Neumann operator $G[\eta]$ is defined by 
    \bq
G[\eta]g=\na_{x, y}\psi(x, \eta(x))\cdot N(x),\quad N(x)=(-\na_x\eta(x), 1),
    \eq
    where $\psi$ solves 
    \bq\begin{cases}
\Delta_{x, y}\psi=0\quad\text{in } \Omega_\eta,\\
\psi(x, \eta(x))=g(x),\\
\p_y\psi(x, -b)=0
    \end{cases}
    \eq
    in the finite depth case, or 
      \bq\begin{cases}
\Delta_{x, y}\psi=0\quad\text{in } \Omega_\eta,\\
\psi(x, \eta(x))=g(x),\\
\na \psi\in L^2(\Omega_\eta)
    \end{cases}
    \eq
in the infinite depth case.
    \end{defi}
It is crucial in our construction of large traveling waves that $G[\eta]$ is invertible for large $\eta$. In the scale of H\"older spaces, this invertibility is stated as follows. 
\begin{prop}[\protect{Proposition 2.5, \cite{Nguyen2023-capillary}}]\label{prop:invertG}
        Let $d\ge 1$, $k\geq 1$, $\alpha\in (0, 1)$, and $\eta \in C^{k, \alpha}(\T^d)$. Then there exists a nondecreasing function $C: \Rr_+\to \Rr_+$,  depending only on $(d, b, \fd, k, \a)$, such that the following assertions hold. \\

(i) $G[\eta]: C^{k, \a}(\T^d)\to C^{k, \a}(\T^d)$ and 
    \bq
\|G[\eta]\|_{C^{k, \a}\to C^{k-1, \alpha}} \leq C(\|\eta\|_{C^{k, \alpha}}).
    \eq  
    (ii) $G[\eta]: \mathring{C}^{k, \alpha}(\T^d) \rightarrow \mathring{C}^{k-1, \alpha}(\T^d)$ is an isomorphism and
\bq\label{est:inverseG}
        \| (G[\eta])^{-1}\|_{\rC^{k-1, \alpha}\to \rC^{k, \alpha}}\le C(\| \eta\|_{C^{k, \alpha}}).
        \eq 
    \end{prop}
    We will implement a fixed-point method to construct traveling waves. The contraction of the nonlinear map will follow from the following contraction estimate for $G[\cdot]$.
    \begin{prop}[\protect{Proposition 2.7, \cite{Nguyen2023-capillary}}]
      Let $d\ge 1$, $k\geq 1$, $\alpha\in (0, 1)$, and $\eta_1$, $\eta_2 \in C^{k, \alpha}(\T^d)$. Then there exists a  nondecreasing function $\wt{C}: \Rr_+\times \Rr_+\to \Rr_+$, depending only on $(d, b, \fd, k, \a)$, such that  
\bq\label{contraction:est:Holder}\|G[\eta_1]g - G[\eta_2]g\|_{C^{k-1, \alpha}} \leq \widetilde C(\|\eta_1\|_{C^{k, \alpha}},\|\eta_2\|_{C^{k, \alpha}}) \|\eta_1 - \eta_2\|_{C^{k, \alpha}}\|g\|_{C^{k, \alpha}}.\eq
    \end{prop}
For the proof of the stability of large traveling waves in Sobolev spaces, we will need the following results on the Dirichlet-Neumann operator in Sobolev spaces. Propositions \ref{prop:DN:Sobolevcontinuity} and \ref{prop:DN:Sobolevcontraction} provide a  continuity estimate and contraction estimate, respectively. 
\begin{prop}[\protect{Theorem 3.12, \cite{ABZ3}}]\label{prop:DN:Sobolevcontinuity}
        Let $d\ge 1$, $s > 1+\frac{d}{2}$, and $\eta \in H^s(\T^d)$.  For any $\sigma \in [\mez, s]$, there exists a nondecreasing function $C: \Rr_+\to \Rr_+$ depending only on $(d, b, \fd, s, \sigma)$ such that
                    \bq\|G[\eta]g\|_{H^{\sigma-1}} \leq C(\|\eta\|_{H^s})  \|g\|_{H^\sigma}
                    \eq
                    for all $g\in H^\sigma(\T^d)$.
    \end{prop}
 %   \begin{theo}
  %      For $s > 1+\frac{d}{2}$, there exists a nondecreasing function $\widetilde C$ of two variables s.t.
   %         \[\|G[\eta_1]g - G[\eta_2]g\|_{H^{s-1}} \leq \widetilde C(\|\eta_1\|_{H^s},\|\eta_2\|_{H^s}) \|\eta_1 - \eta_2\|_{H^s}\|g\|_{H^s}\]
    %\end{theo}  
\begin{prop}\label{prop:DN:Sobolevcontraction}
        Let $d\ge 1$,  $ 1+\frac{d}{2} < s_0 \leq s$, and $\eta_1$, $\eta_2 \in H^s(\T^d)$. There exists a nondecreasing.        function $\widetilde C: \Rr_+\times \Rr_+\to \Rr_+$ depending only on $(d, b, \fd, s, s_0)$ such that     \bq\label{Sobolevcontraction}
        \begin{aligned}
            \|G[\eta_1]g - G[\eta_2]g\|_{H^{s-1}} &\leq \widetilde C(\|\eta_1\|_{H^{s_0}},\|\eta_2\|_{H^{s_0}})\Big\{ \|\eta_1 - \eta_2\|_{H^{s_0}}\|g\|_{H^s} + \|\eta_1 - \eta_2\|_{H^s}\|g\|_{H^{s_0}}\\
            &\hspace{1.6in}+\| \eta_1-\eta_1\|_{H^{s_0}}\|g\|_{H^{s_0}}\big(\| \eta_1\|_{H^s}+\| \eta_2\|_{H^s}\big)\Big\}
        \end{aligned}
        \eq
        for all $g\in H^s(\T^d)$.
    \end{prop}
 When $s=s_0>1+\frac{d}{2}$, the estimate \eqref{Sobolevcontraction} was proven in  Proposition 3.31 in \cite{NgPa}. The {\it tame} version \eqref{Sobolevcontraction} can be obtained by combining the proof of Proposition 3.31 in \cite{NgPa} with tame elliptic estimates in Proposition 2.12 in \cite{Poyferre}.
 
    The next proposition provides coercive estimates for $G[\eta]$, which will be crucial in proving decay of perturbations around large traveling waves.
    \begin{prop}[\protect{Propositions 2.2 and 2.3, \cite{Nguyen2023-coercivity}}]
         Let $\eta\in W^{1, \infty}(\T^d)$, $d\ge 1$. There exists a positive constant $C=C(d)>0$ such that 
         \bq\label{coercive}
(G[\eta]g, g)_{H^{-\mez}(\T^d), H^\mez(\T^d)}\ge M\| g\|_{\dot H^\mez}^2
         \eq
         for all $g\in H^\mez(\T^d)$, where 
         \bq
M=\begin{cases}
\frac{C\fd}{1+\| \na \eta\|_{L^\infty}^2+\|\eta+b\|_{W^{1, \infty}}^2}\quad\text{if } \Omega_\eta~\text{is } \eqref{def:Omega}\\
\frac{C}{1+\| \na \eta\|_{L^\infty}}\quad\text{if } \Omega_\eta~\text{is } \eqref{def:Omega:id}
\end{cases}
         \eq
         and
\bq
\| g\|_{\dot H^\mez(\T^d)}^2:=\sum_{k \in \mathbb{Z}^d\setminus \{0\}}|k|^2|\widehat{g}(k)|^2.
\eq
           \end{prop} 
%%%%%%%%%%%%%%%%%%%%%    
\subsection{Commutator estimates}
We establish commutator estimates for $[\p^\alpha, G[\eta]]$, where   $\alpha\in \mathbb{N}^d$ is any multiindex. Since $G[\eta]$ is a first-order operator, one expects that the commutator $[\p^\alpha, G[\eta]]$ is of order $|\alpha|$, provided $\eta$ is sufficiently smooth. This will be  proven in Theorem \ref{theo:commutator}, the proof of which requires the following lemma.
\begin{lemm}\label{lemm:eestloc} 
We distinguish the finite versus the infinite depth cases.

(i) In the finite depth case \eqref{def:Omega} we consider  the boundary-value problem
\bq\label{eqv:finite}
\begin{cases}
\Delta v=G\quad\text{in } \Omega_{\eta}, \\
v(x, \eta(x))=0,\\
\p_yv(x, -b)=g_b.
\end{cases}
\eq
Let $h>0$ be sufficiently small  such that 
\bq\label{Omegah}
\Omega^h:=\{(x, y)\in \T^d\times \Rr: \eta(x)-h<y<\eta(x)\}\subset \Omega_{\eta}.
\eq
For any $r \in [0,\infty)$, there exists $C: \Rr_+\to \Rr_+$ depending only on $(d, b, \fd, r, h)$ such that 
        \begin{align}\label{eest:loc}
            \|v\|_{H^{r+2}(\Omega^{\frac{h}{2}})} \leq C(\|\eta\|_{W^{\lceil r+2\rceil, \infty}})\left\{\|G\|_{H^r(\Omega^h)} + \|G\|_{L^2(\Omega)} + \|g_b\|_{H_x^{-\frac12}}\right\}
        \end{align}
        provided the right-hand side is finite. 
        
(ii) In the infinite depth case \eqref{def:Omega:id} we consider the problem 
\bq
\begin{cases}
\Delta v=G\quad\text{in } \Omega_{\eta}, \\
v(x, \eta(x))=0,\\
\na v\in L^2(\Omega)
\end{cases}
\eq
For any $r\in [0, \infty)$ and $h>0$, there exists $C: \Rr_+\to \Rr_+$ depending only on $(d, r, h)$ such that
\begin{align}\label{eest:loc:id}
            \|v\|_{H^{r+2}(\Omega^{\frac{h}{2}})} \leq C(\|\eta\|_{W^{\lceil r+2\rceil, \infty}})\left\{\|G\|_{H^r(\Omega^h)} + \|G\|_{L^2(\Omega)} \right\}
        \end{align}
        provided the right-hand side is finite.
\end{lemm}
        \begin{proof}
 %       We shall only prove \eqref{eest:loc} for the finite depth case \eqref{eqv:finite}, which is more difficult. By linear interpolation, it suffices to prove \eqref{eest:loc} for $r\in \{0, 1, 2, \dots\}$.
 %           We have
 %           \begin{align*} 
 %               \int_{\Omega} G \widetilde p &= \int_{\Omega} \widetilde p \Delta \widetilde p \\
  %              &= -\int_{\Omega} |\nabla \widetilde p|^2 - \int_{\mathbb{T}^d} \widetilde p (x,-b) \partial_y \widetilde p(x,-b) \\
 %               &= -\int_{\Omega} |\nabla \widetilde p|^2 - \int_{\mathbb{T}^d} \widetilde p (x,-b) g_b 
 %           \end{align*}
 %           which gives that
 %           \begin{align*}
 %               \int_{\Omega} |\nabla \widetilde p|^2 &= - \int_{\mathbb{T}^d} \widetilde p (x,-b) g_b - \int_{\Omega} G \widetilde p \\
 %               &\leq C(\|\eta\|_{W^{1,\infty}}) \|g_b\|_{H^{-\frac12}}\|\widetilde p\|_{H^1} + \|G\|_{L^2}\|\widetilde p\|_{L^2}.
 %           \end{align*}
 %           Since $\widetilde p (x, \eta(x)) = 0$, the Poincare-Wirtinger Inequality gives
 %           \begin{align*}
 %               \|\widetilde p\|^2_{H^1} \leq C(\|\eta\|_{W^{1,\infty}})\int_{\Omega} |\nabla \widetilde p|^2 \leq C(\|\eta\|_{W^{1,\infty}})(\|g_b\|_{H^{-\frac12}}\|\widetilde p\|_{H^1} + \|G\|_{L^2}\|\widetilde p\|_{H^1}),
 %           \end{align*}
 %           or just
(i) Finite depth.  We have the variational estimate for the Neumann problem \eqref{eqv:finite}:      \begin{align} \label{variest:Neumann}
                \|v\|_{H^1(\Omega)} \leq C(\|\eta\|_{W^{1,\infty}})(\|g_b\|_{H^{-\frac12}_x} + \|G\|_{L^2(\Omega)}).
            \end{align}
See Theorem III.4.3 in \cite{BoyerFabrie}.  On the other hand, elliptic estimates for the Dirichlet problem give
   \bq\label{estv:top}
            \begin{aligned}
                \|v\|_{H^{r+2}(\Omega^{\frac{h}{2}})} &\leq C(\|\eta\|_{W^{\lceil r+2\rceil, \infty}})\left \{ \|G\|_{H^r(\Omega^{\frac{h}{2}})} + \|v |_{y=\eta-\frac{h}{2}}\|_{H^{r+\frac32}_x} \right \} \\
                &\leq C(\|\eta\|_{W^{\lceil r+2\rceil, \infty}})\left \{ \|G\|_{H^r(\Omega^{\frac{h}{2}})} + \|v \|_{H^{r+2}(\{\eta-h<y<\eta-\frac{h}{4}\})} \right \},
            \end{aligned}
            \eq
            where we have used the trace inequality in the second inequality. By a standard localization argument and \eqref{variest:Neumann}, we obtain 
            \bq\label{estv:loc}
            \begin{aligned}
                \|v\|_{H^{r+2}(\{\eta-h<y<\eta-\frac{h}{4}\})} &\leq C(\|\eta\|_{W^{1, \infty}}) \left \{ \|G\|_{H^r(\Omega^h)} + \|v \|_{H^{1}(\Omega^h)} \right \}\\
                &\leq C(\|\eta\|_{W^{1, \infty}}) \left \{ \|G\|_{H^r(\Omega^h)}+ \|G\|_{L^2(\Omega)} + \| g_b\|_{H^{-\mez}_x} \right \}.
            \end{aligned}
            \eq
     Substituting  \eqref{estv:loc} in \eqref{estv:top} yields the desired estimate \eqref{eest:loc}.

    (ii) Infinite depth. In this case we have 
    \[
     \|\na v\|_{L^2(\Omega)} \leq C(\|\eta\|_{W^{1,\infty}}) \|G\|_{L^2(\Omega)}.
    \]
    Moreover, since $v(x, \eta(x))=0$, $v$ satisfies Poincare's inequality in the strip $\Omega^h$. Consequently, the $H^{r+2}$  estimate \eqref{estv:top} and the interior estimate \eqref{estv:loc} hold without the term $\| g_b\|_{H^{-\mez}_x}$. Therefore, we obtain \eqref{eest:loc:id}.
        \end{proof}
    \begin{theo}\label{theo:commutator}
        Let $\alpha\in \mathbb{N}^d$. For $\Rr\ni \sigma \geq \frac{1}{2}$, there exists $C:\Rr_+\to \Rr_+$ depending only on $(d, b, \fd, \sigma, |\alpha|)$ such that the commutator $[\partial^\alpha, G[\eta]]:= \partial^\alpha G[\eta]  -  G[\eta]\partial^\alpha $ satisfies
\bq\label{commutator:est}
            \|[\partial^\alpha, G[\eta]]f\|_{H^\sigma(\T^d)} \leq C(\|\eta\|_{W^{\lceil\sigma+|\alpha|+\frac32\rceil, \infty}(\T^d)})\|f\|_{H^{\sigma+|\alpha|}(\T^d)}
        \eq
        provided the right-hand side is finite.
    \end{theo}
    \begin{proof}
        We will only consider the more difficult case of finite depth. We will prove \eqref{commutator:est} by induction on $|\alpha|\in \mathbb{N}$. We first focus on the case $|\alpha| = 1$, i.e., $\p^\alpha=\p_j\equiv \frac{\p}{\p x_j}$ for some $j\in\{1, \dots, d\}$.  Thus we consider $\partial_j G[\eta]f - G[\eta] \partial_j f$, where $f\in H^{\sigma+1}(\T^d)$ with $\sigma \ge \mez$. Let $q$ and $p$ solve the problem
            \bq\label{ellitpic:q}
            \begin{cases}
            \Delta_{x,y} q = 0 \text{ in } \Omega,\\
            q |_{y=\eta} = f,\\
            \partial_y q |_{y=-b} = 0
            \end{cases}
            \eq
        and 
            \bq\label{ellitpic:p}
            \begin{cases}\Delta_{x,y} p = 0 \text{ in } \Omega,\\
            p |_{y=\eta} = \partial_j f,\\
            \partial_y p |_{y=-b} = 0,
            \end{cases}
            \eq
        respectively. Then we have 
      \[
            G[\eta] \partial_j f = \nabla p(x, \eta(x)) \cdot (-\nabla \eta, 1) 
            \]
            and
            \bq\label{djG}
            \begin{aligned}
            \partial_j G[\eta]f(x) %&= \partial_j \nabla q(x, \eta(x)) \cdot (-\nabla \eta, 1) \\
            &= \partial_j \left\{ -\nabla_x q(x, \eta(x))\cdot\nabla \eta + \partial_y q(x, \eta(x)) \right\} \\
            &= -\nabla \eta\cdot\left\{\partial_j\nabla_x q(x, \eta(x)) + \partial_y \nabla_x q(x, \eta(x))\partial_j \eta\right\} \\
            & \qquad + \partial^2_{jy} q(x, \eta(x)) + \partial^2_y q(x, \eta(x))\partial_j \eta(x)\\
            &\qquad - \partial_j \nabla \eta\cdot\nabla_x q(x, \eta(x)).
        \end{aligned}
        \eq
       A key idea is to introduce the combination 
        \bq
            h(x, y) = \partial_j q(x,y) + \partial_j \eta(x) \partial_y q(x,y).
        \eq
        We first note that
        \bq\label{trace:h}
            h(x, \eta(x)) = \frac{\p}{\p {x_j}} [q(x,\eta(x))] = \partial_j f(x)= p(x, \eta(x)),
        \eq
        so $h$ and $p$ have the same trace on $\{y=\eta(x)\}$. We will show that $h$ is a good approximation of $p$. We compute 
        \bq\label{Delta:h}
            \Delta_{x,y} h = \partial_j \Delta_x \eta\partial_y q + 2 \partial_j \nabla_x \eta\cdot\partial_y \nabla_x q 
        \eq
        and
    \begin{align}\label{naxh}
        \nabla_x h &= \partial_j \nabla_x q +  \partial_j \nabla \eta \partial_y q+ \partial_j \eta \partial_y \nabla_x q,\\ \label{pyh}
\p_yh&=\p^2_{jy}q+\p_j\eta\p^2_yq=\p^2_{jy}q-\p_j\eta\Delta_xq.
        \end{align}
        Since $\p_yq(x, -b)=0$, it follows that 
        \bq\label{trace:pyh}
            \partial_y h(x,-b)  = -\partial_j \eta(x)\Delta_x q(x,-b).
        \eq
     Moreover, we have
     \bq\label{dNh}
        \begin{aligned}
            \nabla h(x, \eta(x)) \cdot (-\nabla \eta, 1) &= -\nabla \eta(x) \cdot \left\{\partial_j \nabla_x q(x,\eta(x)) +  \partial_y \nabla_x q(x,\eta(x)) \partial_j \eta(x) \right\}\\
            &\qquad+ \partial_{jy}^2 q(x,\eta(x)) +  \partial_y^2 q (x,\eta(x))\partial_j \eta (x)\\
            & \qquad +\nabla \eta(x) \cdot  \partial_j \nabla \eta(x)\partial_y q(x,\eta(x)).
        \end{aligned}
        \eq
        Comparing \eqref{djG} and \eqref{dNh} we see that their right-hand sides have exactly the same second order terms with respect to $q$. Consequently 
        \bq\label{dG-h}
        \begin{aligned}
            \partial_j G[\eta]f - \nabla h(x, \eta(x)) \cdot (-\nabla \eta, 1) &= \nabla \eta(x) \cdot \partial_j \nabla \eta(x)\partial_y q(x,\eta(x))  - \partial_j \nabla \eta(x)\cdot\nabla_x q(x, \eta(x)).
        \end{aligned}
        \eq
        Now we set
        \[
            \widetilde p = p - h.
        \]
        In view of \eqref{ellitpic:q}, \eqref{trace:h}, and \eqref{trace:pyh}, we find that $\wt{p}$ satisfies the boundary-value problem
        \bq\label{sys:wtp}
        \begin{cases}
            \Delta_{x,y} \widetilde p =  -\partial_j \Delta_x \eta\partial_y q - 2 \partial_j \nabla_x \eta\cdot\partial_y \nabla_x q=: G \text{ in } \Omega,\\
            \widetilde p(x,\eta(x)) = 0,\\
            \partial_y \widetilde p(x,-b) = \Delta_x q(x,-b)\partial_j \eta(x) =: g_b.
        \end{cases}
        \eq
        Approximating $p$ by $h$, we deduce from \eqref{dG-h} that
        \begin{align*}
            \partial_j G[\eta]f - G[\eta] \partial_j f &= \partial_j G[\eta]f - \nabla p(x, \eta(x)) \cdot (-\nabla \eta, 1) \\
            &= \partial_j G[\eta]f - \nabla h(x, \eta(x)) \cdot (-\nabla \eta, 1) - \nabla \widetilde p(x, \eta(x)) \cdot (-\nabla \eta, 1) \\
            &= \nabla \eta(x) \cdot \partial_j \nabla \eta(x)\partial_y q(x,\eta(x))  - \partial_j \nabla \eta(x)\cdot\nabla_x q(x, \eta(x))\\
            &\qquad- \nabla \widetilde p(x, \eta(x)) \cdot (-\nabla \eta, 1).
        \end{align*}
        From here, we can  begin to  estimate  the $H^\sigma$ norm of $[\p_j, G[\eta]]f$. We will appeal to the following product estimate 
        \bq\label{productest}
            \|uv\|_{H^r(\mathbb{T}^d)} \leq C(r, d)\|u\|_{H^r(\mathbb{T}^d)}\|v\|_{W^{\lceil r\rceil, \infty}(\mathbb{T}^d)},\quad r\in \Rr.
        \eq
        For $r\in \mathbb{N}$, \eqref{productest} follows  from Leibniz's rule. Since the mapping $u\mapsto uv$ is linear, the general case $r\in (0, \infty)$ follows from interpolation, and the general case $r\in (-\infty, 0)$ follows from the fact that  $H^{-r}(\T^d)$ is the dual space of $H^r(\T^d)$.

        By virtue of \eqref{cd:fd},  \eqref{Omegah} holds for $h>0$ sufficiently small. For $\sigma > 0$, we have the continuity  $H^{\sigma+\mez}_{x, y}(\Omega^\frac{h}{2})\to H^\sigma_x(\T^d)$ of the trace operator. Combining this with \eqref{productest}, we obtain
        \bq\label{est:comm:1}
        \begin{aligned}
            \|\partial_j G[\eta]f - G[\eta] \partial_j f\|_{H^\sigma} &\leq \|\nabla \eta(x) \cdot \partial_j \nabla \eta(x)\partial_y q(x,\eta(x))\|_{H_x^\sigma} + \|\partial_j \nabla \eta(x)\cdot\nabla_x q(x, \eta(x))\|_{H_x^\sigma} \\
            &\qquad+ \|\nabla \widetilde p(x, \eta(x)) \cdot (-\nabla \eta, 1)\|_{H_x^\sigma} \\
            &\leq C(\|\eta\|_{W^{\lceil\sigma+2\rceil, \infty}}) \|\na q(x, \eta(x))\|_{H_x^\sigma}  + C(\|\eta\|_{W^{\lceil\sigma+1\rceil, \infty}}) \|\nabla \widetilde p(x, \eta(x))\|_{H_x^\sigma} \\
            &\leq C(\|\eta\|_{W^{\lceil\sigma+2\rceil, \infty}}) \left\{\|\na q\|_{H_{x,y}^{\sigma+\frac{1}{2}}(\Omega^\frac{h}{2})} + \|\nabla \widetilde p\|_{H_{x,y}^{\sigma+\frac{1}{2}}(\Omega^\frac{h}{2})}\right\} \\
            &\leq C(\|\eta\|_{W^{\lceil\sigma+2\rceil, \infty}}) \left\{ \|\na q\|_{H_{x,y}^{\sigma+\mez}(\Omega^\frac{h}{2})} + \|\widetilde p\|_{H_{x,y}^{\sigma+\frac{3}{2}}(\Omega^\frac{h}{2})}\right\}.
        \end{aligned}
        \eq
        Since $\wt{p}$ solves the problem \eqref{sys:wtp}, we can apply  Lemma \ref{lemm:eestloc} (i) with $r=\sigma+\frac32\ge 2$ (for $\sigma\ge \mez$) and invoke \eqref{productest} to have
        \bq\label{est:wtp}
        \begin{aligned}
            \|\widetilde p\|_{H_{x,y}^{\sigma+\frac{3}{2}}(\Omega^\frac{h}{2})} &\leq C(\|\eta\|_{W^{\lceil\sigma+\frac32\rceil, \infty}})\{\|G\|_{H_{x,y}^{\sigma-\frac{1}{2}}(\Omega^h)} + \|g_b\|_{H_x^{-\frac12}}\} \\
            &\leq C(\|\eta\|_{W^{\lceil\sigma+\frac32\rceil, \infty}})\left\{\|\partial_j \Delta_x \eta\partial_y q\|_{H_{x,y}^{\sigma-\frac{1}{2}}(\Omega^h)} + \|2 \partial_j \nabla_x \eta\cdot\partial_y \nabla_x q\|_{H_{x,y}^{\sigma-\frac{1}{2}}(\Omega^h)} \right.\\
            & \quad \left.+ \|\Delta_x q(\cdot,-b) \partial_j \eta\|_{H_x^{-\frac12}}\right\} \\
            &\leq C(\|\eta\|_{W^{\lceil\sigma+\frac52\rceil, \infty}}) \{\|\partial_y q\|_{H_{x,y}^{\sigma-\frac{1}{2}}(\Omega^h)} + \|\partial_y \nabla_x q\|_{H_{x,y}^{\sigma-\frac{1}{2}}(\Omega^h)} + \|\Delta_x q(\cdot,-b)\|_{H_x^{-\frac12}}\} \\
            &\leq C(\|\eta\|_{W^{\lceil\sigma+\frac52\rceil, \infty}})\{ \|\na q\|_{H_{x,y}^{\sigma+\frac{1}{2}}(\Omega^h)} + \|q(\cdot,-b)\|_{H_x^{\frac32}}\} \\
            &\leq C(\|\eta\|_{W^{\lceil\sigma+\frac52\rceil, \infty}})\{ \|q\|_{H_{x,y}^{\sigma+\frac{3}{2}}(\Omega)}+\|q\|_{H_{x,y}^2(\Omega)}\} \\
            &\leq C(\|\eta\|_{W^{\lceil\sigma+\frac52\rceil, \infty}})\|q\|_{H_{x,y}^{\sigma+\frac{3}{2}}(\Omega)},
        \end{aligned}
        \eq
        where in the two last estimates we have used the trace inequality $\|q(\cdot,-b)\|_{H_x^{\frac32}}\le C\|q\|_{H_{x,y}^2(\Omega)}$ and the condition $\sigma+\frac32\ge 2$. 
        Plugging \eqref{est:wtp} back in \eqref{est:comm:1} yields
        \[
            \|\partial_j G[\eta]f - G[\eta] \partial_j f\|_{H^\sigma} \leq C(\|\eta\|_{W^{\lceil\sigma+\frac52\rceil, \infty}}) \|q\|_{H_{x,y}^{\sigma+\frac{3}{2}}(\Omega)}.
        \]
        Then, invoking the elliptic estimate 
        \[     \|q\|_{H_{x,y}^{\sigma+\frac{3}{2}}(\Omega)}\le C(\|\eta\|_{W^{\lceil\sigma+\frac32\rceil, \infty}})\| f\|_{H^{\sigma+1}_x}
        \]
        for the problem \eqref{ellitpic:q}, we obtain the desired commutator estimate 
    \bq\label{commest:base}
            \|\partial_j G[\eta]f - G[\eta] \partial_j f\|_{H^\sigma_x} \leq C(\|\eta\|_{W^{\lceil\sigma+\frac52\rceil, \infty}}) \|f\|_{H_x^{\sigma+1}}.
        \eq
        With the base case $|\alpha| = 1$ in hand, suppose  that the commutator estimate \eqref{commutator:est}  holds for $|\alpha| \le  k-1$, $k\ge 2$. Then for $|\alpha| = k$, we write $\alpha = \partial_j \partial^\beta$ for some $\beta$ with $|\beta| \le k-1$, so that
        \begin{align*}
            \partial_j[\partial^\beta, G[\eta]] &= \partial^\alpha G[\eta]- \partial_j G[\eta]\partial^\beta   \\
            &= \partial^\alpha G[\eta]-G[\eta]\partial^\alpha+ G[\eta]\partial_j\partial^\beta- \partial_j G[\eta]\partial^\beta    \\
            &=  [\partial^\alpha, G[\eta]] +[G[\eta], \partial_j]\partial^\beta.
        \end{align*}
        Thus using the base estimate \eqref{commest:base} and the induction hypothesis, we deduce
        \begin{align*}
            \|[G[\eta], \partial^\alpha]f\|_{H^\sigma_x} &\leq \|\partial_j[G[\eta], \partial^\beta] f\|_{H^\sigma_x} + \|[G[\eta], \partial_j]\partial^\beta f\|_{H^\sigma_x} \\
            & \leq \|[G[\eta], \partial^\beta] f\|_{H^{\sigma+1}_x} + C(\|\eta\|_{W^{\lceil\sigma+\frac52\rceil, \infty}})\|\partial^\beta f\|_{H^{\sigma+1}_x} \\
            &\leq  C(\|\eta\|_{W^{\lceil\sigma+1+|\beta|+\frac32\rceil, \infty}})\|f\|_{H^{\sigma+1+|\beta|}_x}+ C(\|\eta\|_{W^{\lceil\sigma+\frac52\rceil, \infty}})\|f\|_{H^{\sigma+1+|\beta|}_x} \\
            & \leq C(\|\eta\|_{W^{\lceil\sigma+|\alpha|+\frac32\rceil, \infty}})\|f\|_{H^{\sigma+|\alpha|}_x}.
        \end{align*}
        The proof of \eqref{commutator:est} is complete.
    \end{proof}
    \begin{rema}
      (i)  By virtue of the pointwise cancellations in the preceding proof, Theorem \ref{theo:commutator} can be proven analogously in other variants such as the nonperiodic setting  $f\in H^s(\Rr^d)$, commutator estimates in other norms (e.g. H\"older), and nonflat bottoms $\{y=-b(x)\}$.

      (ii) The commutator estimate \eqref{commutator:est}  shows a full gain of one derivative and its proof  is based entirely on physical space. On the other hand,  the {\it paralinearization} results in  \cite{ABZ1, NgPa} would only yield a gain of $1/2$ derivative.  
    \end{rema}
%%%%%%%%%%%%%% nonflat bottom calculations
%\begin{rema}
% Theorem \ref{theo:commutator} remains valid when the bottom is not flat but given by $\{(x, -b(x)): x\in \T^d)$. In this case, we have the homogeneous Neumann condition for $q$ and $p$:
% \[
%\p_\nu q\vert_{y=-b(x)}=0=\p_\nu p\vert_{y=-b(x)},\quad \nu:=(-\na b, -1).
% \]
% Differentiating the equation $\p_\nu q(x, -b(x))=0$ in $x_j$ yields 
%  \[
%  \p_j\na b\cdot \na_xq+\na_xb\cdot \p_j\na_xq+\p^2_{jy}q-\p^2_{yy}q\p_jb=0 \quad\text{on~}\{y=-b(x)\}.
%  \]
%  But $q$ is harmonic, so 
%  \bq\label{pjyq:nf}
%  \p^2_{jy}q=-\p_j\na b\cdot \na_xq-\na_xb\cdot \p_j\na_xq-\Delta_xq\p_jb\quad\text{on~}\{y=-b(x)\}.
%  \eq
%  combining \eqref{pyh} and \eqref{pjyq:nf} gives
%  \bq\label{pyh:nf}
%\p_yh=-\p_j\na b\cdot \na_xq-\na_xb\cdot \p_j\na_xq-\Delta_xq(\p_jb+\p_j\eta)\quad\text{on~}\{y=-b(x)\}.
%  \eq
%  Next, using \eqref{naxh} and the fact that $\p_yq=-\na b\cdot \na_xq$ on $\{y=-b(x)\}$, we obtain 
%  \bq\label{naxh:nf}
%  \nabla_x h = \partial_j \nabla_x q -  \partial_j \nabla \eta (\na b\cdot \na_xq)+ \partial_j \eta \partial_y \nabla_x q\quad\text{on~}\{y=-b(x)\}.
%  \eq
%  It follows from \eqref{pyh:nf} and \eqref{naxh:nf} that 
%  \bq\label{dnup:nf}
%  \begin{aligned}
%\p_\nu \wt{p}\vert_{y=-b(x)}=-\p_\nu h\vert_{y=-b(x)}&=\Big\{-(\na b\cdot \partial_j \nabla \eta)(\na b\cdot \na_xq)+ \partial_j \eta \na b\cdot \partial_y \nabla_x q\\
%&\quad-\p_j\na b\cdot \na_xq-\Delta_xq(\p_jb+\p_j\eta)\Big\}\vert_{y=-b(x)}.
%  \end{aligned}
%  \eq
%  \end{rema}
\section{Existence and uniqueness of slowly traveling  waves}\label{sec:existence}
    Given an external pressure $\phi: \T^d\to \Rr$, we shall construct traveling waves $\eta$ with speed $\gamma$  as solutions of  \eqref{eq:tw}.
    \begin{theo}
        Let $k \geq 1$, $\alpha \in (0,1)$, and $\phi \in \rC^{k,\alpha}(\mathbb{T}^d)$. In the finite depth case we assume that    \bq\label{separation:phi}
        \mu(\phi):=\inf_{x\in \T^d} (-\phi(x)+b)>0.
        \eq
         There exists a small number $\delta_0 = \delta_0(\|\phi\|_{C^{k,\alpha}})>0$ such that the following holds. For all $\delta \in (0, \delta_0)$, there exists $\eps = \eps(\delta, \|\phi\|_{C^{k,\alpha}}) > 0$ such that if $|\gamma| < \eps$ then \eqref{eq:tw} has a unique solution $\eta \in \rC^{k, \alpha}(\mathbb{T}^d)$ satisfying $\|\eta+\phi\|_{C^{k, \alpha}} \le  \delta$. Moreover, for any $\delta\in (0, \delta_0)$, the mapping
\bq\label{map:gammaeta}
         (-\eps, \eps)\ni \gamma \mapsto \eta\in \overline{B}_\delta(-\phi)\subset \rC^{k, \a}(\T^d)
         \eq
         is Lipschitz continuous. 
    \end{theo}
    \begin{proof}
    The idea is to perturb around the solution \eqref{origin}. To this end, we set $\zeta = \eta + \phi$, so that  \eqref{eq:tw} can be equivalently rewritten as 
        \bq\label{eq:tw:zeta}
        \begin{aligned}
        0 &= -\gamma\partial_1\zeta + \gamma \partial_1\phi+G[\zeta-\phi]\zeta\\
        &= -\gamma\partial_1\zeta + \gamma \partial_1\phi+G[-\phi]\zeta + (G[\zeta-\phi]\zeta-G[-\phi]\zeta).
        \end{aligned}
        \eq
        In the finite depth case, we assume that $\| \zeta\|_{C(\T^d)}<\mu(\phi)$  so that  \eqref{separation:phi} yields $\zeta-\phi>-b$ in order for $G[\zeta-\phi]$ to be well-defined. By assumption, $\phi \in C^{k, \alpha}$, so we have by Theorem 2.8 that $G[-\phi]:\mathring{C}^{k,\alpha} \rightarrow \mathring{C}^{k-1,\alpha}$ is an isomorphism. We can therefore invert $G[-\phi]$ in $G[-\phi]\zeta$ and rewrite \eqref{eq:tw:zeta} as
        \bq\label{T:zeta}
        \zeta = (G[-\phi])^{-1} \left\{\gamma\partial_1\zeta -\gamma\partial_1\phi-\left(G[\zeta-\phi]\zeta-G[-\phi]\zeta\right)\right\} =: T_\gamma(\zeta).\eq
    This reformulation  is a fixed point problem in $\zeta\in \rC^{k, \alpha}(\mathbb{T}^d)$. Our goal will be to find small posotive numbers $\delta_0$ and $\eps$ such that $T_\gamma$ is a contraction mapping on any ball $\overline{B}_\delta(0)$ in $\rC^{k, \alpha}(\mathbb{T}^d)$ provided $\delta<\delta_0$ and $|\gamma|<\eps$. 
    
    Suppose $\zeta$ is in $\overline{B}_\delta(0)$ for some $0 < \delta < 1$. We shall only consider the infinite depth case as the finite depth case only requires the  additional condition $0<\delta<\mu(\phi)$.  Combining the estimate \eqref{est:inverseG} for the norm of $(G[-\phi])^{-1}$ and the contraction estimate \eqref{contraction:est:Holder} for $G[\cdot]$, we obtain
    \begin{align*}
   \|T_\gamma(\zeta)\|_{C^{k, \alpha}} &\leq C(\|\phi\|_{C^{k,\alpha}}) \|\gamma\partial_1\zeta -\gamma\partial_1\phi-[G[\zeta-\phi]\zeta-G[-\phi]\zeta]\|_{C^{k-1, \alpha}} \\
         &\leq C(\|\phi\|_{C^{k,\alpha}}) \left\{|\gamma|\|\zeta\|_{C^{k, \alpha}} + |\gamma|\|\phi\|_{C^{k, \alpha}} + \widetilde C(\|\zeta-\phi\|_{C^{k, \alpha}}, \|\phi\|_{C^{k, \alpha}})\|\zeta\|^2_{C^{k, \alpha}}\right\} \\
        &\leq C(\|\phi\|_{C^{k,\alpha}})\left\{|\gamma|\delta+|\gamma|\|\phi\|_{C^{k, \alpha}}+\widetilde C(\|\phi\|_{C^{k, \alpha}}+1, \|\phi\|_{C^{k, \alpha}})\delta^2\right\} \\
        &\leq A|\gamma|\delta+A|\gamma|\|\phi\|_{C^{k, \alpha}}+A\delta^2,
    \end{align*}
    where $A=A(\|\phi\|_{C^{k,\alpha}})$. If $\delta$ and $\gamma$ satisfy 
    \bq\label{cd:delta:1}
    \delta < \frac{1}{4A},\quad |\gamma| < \min(\frac{1}{4A}, \frac{\delta}{4A\|\phi\|_{C^{k, \alpha}}}),
    \eq
    then $T_\gamma$ maps $\overline{B}_\delta(0)$ into itself.
    
   Next, we show that $T_\gamma(\zeta): \overline{B}_\delta(0) \rightarrow \overline{B}_\delta(0)$ is  a contraction. Suppose $\zeta_1$, $ \zeta_2 \in \overline{B}_\delta(0)$. Then \eqref{est:inverseG} implies 
   \bq\label{contraction:T:1}
 \begin{aligned}
       & \|T_\gamma(\zeta_1)-T_\gamma(\zeta_2)\|_{C^{k, \alpha}} \\
        &\leq C(\|\phi\|_{C^{k,\alpha}})\|\gamma\partial_1\zeta_1-\gamma\partial_1\zeta_2+G[\zeta_2-\phi]\zeta_2-G[-\phi]\zeta_2-G[\zeta_1-\phi]\zeta_1+G[-\phi]\zeta_1\|_{C^{k-1, \alpha}} \\
        &\leq C(\|\phi\|_{C^{k,\alpha}})\left\{|\gamma|\|\zeta_1-\zeta_2\|_{C^{k, \alpha}}+\|G[\zeta_2-\phi]\zeta_2-G[-\phi]\zeta_2-G[\zeta_1-\phi]\zeta_1+G[-\phi]\zeta_1\|_{C^{k-1, \alpha}}\right\}.
\end{aligned}
\eq
    Focusing on the terms involving the Dirichlet-Neumann operator $G$, we write
\begin{align*}&G[\zeta_2-\phi]\zeta_2-G[-\phi]\zeta_2-G[\zeta_1-\phi]\zeta_1+G[-\phi]\zeta_1 \\
        &= -G[\zeta_2-\phi](\zeta_1-\zeta_2)+G[\zeta_2-\phi]\zeta_1+G[-\phi](\zeta_1-\zeta_2)-G[\zeta_1-\phi]\zeta_1\\
        &=\left\{G[-\phi](\zeta_1-\zeta_2)-G[\zeta_2-\phi](\zeta_1-\zeta_2)\right\}+\left\{G[\zeta_2-\phi]\zeta_1-G[\zeta_1-\phi]\zeta_1\right\}.\end{align*}
    The contraction estimate \eqref{contraction:est:Holder} then yields
        \begin{align*}
            &\|G[\zeta_2-\phi]\zeta_2-G[-\phi]\zeta_2-G[\zeta_1-\phi]\zeta_1+G[\zeta_1-\phi]\zeta_1\|_{C^{k-1, \alpha}} \\
        &\leq\|G[-\phi](\zeta_1-\zeta_2)-G[\zeta_2-\phi](\zeta_1-\zeta_2)\|_{C^{k-1, \alpha}}+\|G[\zeta_2-\phi]\zeta_1-G[\zeta_1-\phi]\zeta_1\|_{C^{k-1, \alpha}} \\
        &\leq\widetilde C(\|-\phi\|_{C^{k, \alpha}}, \|\zeta_2-\phi\|_{C^{k, \alpha}})\|\zeta_2\|_{C^{k, \alpha}} \|\zeta_1-\zeta_2\|_{C^{k, \alpha}}\\
        &\quad+ \widetilde C(\|\zeta_2-\phi\|_{C^{k, \alpha}}, \|\zeta_1-\phi\|_{C^{k, \alpha}})\|\zeta_1-\zeta_2\|_{C^{k, \alpha}}\|\zeta_1\|_{C^{k, \alpha}} \\
        &\leq 2\delta\widetilde C(1+\|\phi\|_{C^{k, \alpha}}, 1+\|\phi\|_{C^{k, \alpha}}) \|\zeta_1-\zeta_2\|_{C^{k, \alpha}}.
       % &:= 2E\delta\|\zeta_1-\zeta_2\|_{C^{k, \alpha}},
    \end{align*}
  %  where $E=E(\|\phi\|_{C^{k,\alpha}})$.
  Putting this back into \eqref{contraction:T:1}, we find
\bq\label{contraction:Tgamma}
\|T_\gamma(\zeta_1)-T_\gamma(\zeta_2)\|_{C^{k, \alpha}} 
        \leq B|\gamma|\|\zeta_1-\zeta_2\|_{C^{k, \alpha}}+ B\delta\|\zeta_1-\zeta_2\|_{C^{k, \alpha}},
\eq
    where  $B=B(\|\phi\|_{C^{k,\alpha}})$. If $\delta$ and $\gamma$ satisfy 
    \bq\label{cd:delta:2}
    \delta < \frac{1}{4B},\quad |\gamma| < \frac{1}{4B},
    \eq
    then $T_\gamma$ is a contraction mapping on $\overline{B}_\delta(0)$. 
    
    In view of \eqref{cd:delta:1} and \eqref{cd:delta:2}, we conclude that if
\bq\label{parameters:existence}
    \delta < \min(\frac{1}{4A}, \frac{1}{4B})=:\delta_0(\| \phi\|_{C^{k, \alpha}}),\quad |\gamma|<\min(\frac{1}{4A}, \frac{\delta}{4A\|\phi\|_{C^{k, \alpha}}}, \frac{1}{4B})=:\eps(\delta, \|\phi\|_{C^{k, \alpha}})
    \eq
    then $T_\gamma$ has a unique fixed point in the $\overline{B}_\delta (0)$ in $\rC^{k, \alpha}(\mathbb{T}^d)$. Put another way, there exists $\delta_0=\delta_0(\| \phi\|_{C^{k, \alpha}}) > 0$ such that the following holds: for all $0<\delta<\delta_0$ there exists $\eps =\eps(\delta, \|\phi\|_{C^{k, \alpha}})> 0$ such that for all $|\gamma| < \eps$ there is a unique traveling wave $\eta\in  \overline{B}_\delta(-\phi)\subset \rC^{k, \alpha}(\mathbb{T}^d)$ with speed $\gamma$.

    Now, we fix $\delta\in (0, \delta_0)$ and consider $\gamma_1$, $\gamma_2\in (-\eps, \eps)$, where $\eps=\eps(\delta, \|\phi\|_{C^{k, \alpha}})$ as given above. Let $\zeta_j=\eta_j+\phi$ be the fixed point of $T_{\gamma_j}$ in $\overline{B}_\delta(0)$. Then \eqref{T:zeta} implies
    \bq\label{zeta:diff}
    \zeta_1-\zeta_2=\left(T_{\gamma_1}(\zeta_1)-T_{\gamma_1}(\zeta_2)\right)+\left(T_{\gamma_1}(\zeta_2)-T_{\gamma_2}(\zeta_2)\right),
    \eq
    where $\zeta_1-\zeta_2=\eta_1-\eta_2$. Using \eqref{contraction:Tgamma} with $\gamma=\gamma_1$ gives
    \bq\label{zeta:diff2}
\|T_{\gamma_1}(\zeta_1)-T_{\gamma_1}(\zeta_2)\|_{C^{k, \alpha}} 
        \leq B|\gamma_1|\|\zeta_1-\zeta_2\|_{C^{k, \alpha}}+ B\delta\|\zeta_1-\zeta_2\|_{C^{k, \alpha}}.
        \eq
        On the other hand, by applying \eqref{est:inverseG} and increasing $B=B(\| \phi\|_{C^{k, a}})$ if necessary, we find 
        \bq\label{zeta:diff3}
        \begin{aligned}
\| T_{\gamma_1}(\zeta_2)-T_{\gamma_2}(\zeta_2)\|_{C^{k, \a}}&=\|(G[-\phi])^{-1}\left\{(\gamma_1-\gamma_2)\p_1\zeta_2-(\gamma_1-\gamma_2)\p_1\phi \right\}\|_{C^{k, \a}}\\
& \le B|\gamma_1-\gamma_2|\| \zeta_2\|_{C^{k, a}}+B|\gamma_1-\gamma_2|\| \phi\|_{C^{k, \a}}\\
&\le B\delta|\gamma_1-\gamma_2|+B|\gamma_1-\gamma_2|\| \phi\|_{C^{k, \a}}.
\end{aligned}
        \eq
        Combining \eqref{zeta:diff}, \eqref{zeta:diff2}, \eqref{zeta:diff3}, and recalling the choice \eqref{parameters:existence} of $\delta$ and $\eps$, we deduce
        \[
\| \zeta_1-\zeta_2\|_{C^{k, \a}}\le \mez \| \zeta_1-\zeta_2\|_{C^{k, \a}}+\left(\frac14+B\| \phi\|_{C^{k, \a}}\right)|\gamma_1-\gamma_2|.
        \]
        We thus obtain
        \[
        \| \zeta_1-\zeta_2\|_{C^{k, \a}}\le \left(\frac12+2B\| \phi\|_{C^{k, \a}}\right)|\gamma_1-\gamma_2|,
        \]
        thereby concluding the Lipschitz continuity of the mapping \eqref{map:gammaeta}.
    \end{proof}
    %%%%%%%%%%%%%%%%%%%%%%%
   \section{Linear dynamics near large traveling waves}\label{sec:lin}
   Let $\eta_*: \T^d\to \Rr$ be a periodic traveling wave with speed $\gamma$. We study the linear dynamics generated by linearization of \eqref{eq:eta} around $\eta_*$. We will not make any  smallness  assumption on $\eta_*$. The linear operator of interest is 
   \bq\label{def:cL}
    \cL = \gamma\partial_1  -G[\eta_*],\quad \gamma \in \Rr.
\eq
$\cL$ is the sum of a skew-adjoint and a self-adjoint operator. 

Our goal is to establish the well-posedness of the linear evolution equation associated to $\cL$. This will be achieved in the Banach space 
\bq\label{def:Xs}
    X_T^s \equiv X^s([0, T])= C([0,T]; \rH^s(\T^d)) \cap L^2([0,T]; H^{s+\frac12}(\T^d))
\eq
 equipped with the norm
\bq
    \|f\|_{X^s_T} = \|f\|_{C([0,T]; H^s)} + \|f\|_{L^2([0,T]; H^{s+\frac12})}.
\eq
   \begin{prop}\label{prop:g}
    Let $s\in \mathbb{N}$, $\eta_*\in W^{s+2, \infty}(\T^d)\cap H^r(\T^d)$ with $r\ge s+1$ and $r>1+\frac{d}{2}$. Then for any $g_0\in \rH^s(\T^d)$, $T>0$, and $F\in L^2([0, T]; \rH^{s-\mez}(\T^d))$, there exists a unique solution $g\in X^s_T$ to the initial-value problem 
  \bq\label{eq:g:2}
    \partial_t g = \cL g+F,\quad g\vert_{t=0} = g_0,
\eq
Moreover,  we have 
    \bq\label{est:g:Xs}
    \|g\|_{X^s_T} \leq C(\|\eta_*\|_{W^{s+2, \infty}}) \left\{\|g_0\|_{H^s}+\| F\|_{L^2([0, T]; H^{s-\mez})}\right\}
    \eq
    for some $C: \Rr_+\to \Rr_+$ depending only on $(s,  d, b,  \fd)$.
\end{prop}
We will prove Proposition \ref{prop:g} by the method of vanishing viscosity. This  requires well-posedness of the regularized problem, established in the following lemma.
\begin{lemm}\label{lemm:geps}
Let $r>1+\frac{d}{2}$, $-\mez\le s\le r-1$, and $\eta_*\in H^r(\T^d)$. For $\eps \in (0, 1)$, we set 
\[
L_\eps=\gamma\p_1+\eps \Delta.
\]
For any $g_0\in \rH^s(\T^d)$, $T>0$, and $F\in L^2([0, T]; \rH^s(\T^d))$, there exists a unique solution $g\in Y^s([0, T])$ to the initial-value problem 
\bq\label{modified}
    \partial_t g^\eps= L_\eps g^\eps - G[\eta_*]g^\eps+F,
    \quad g^\eps|_{t=0} = g_0,
\eq 
where $Y^s([0, T])$ is the Banach space
\bq
    Y^s([0, T]):= C([0,T]; \rH^s) \cap L^2([0,T]; H^{s+1}) \subset X^s_T.
\eq
\end{lemm}
\begin{proof}
We equip $Y^s([0, T])$ with the norm 
\bq\label{def:Ynorm}
\| u\|_{Y^s([0, T])}=\| u\|_{C([0, T]; H^s)}+\eps^\mez \| u\|_{L^2([0, T]; H^{s+1})}.
\eq
 For a given $u\in Y^s([0, T])$, we let $j$ solve
\bq\label{eq:j}
    \partial_t j = L_\eps j,\quad j|_{t=0} = g_0,
\eq
and for any time $\tau\in [0, T]$, we  let $k_\tau(x, t)$ solve
\bq\label{eq:k}
    \partial_t k_\tau = L_\eps k_\tau,\quad k_\tau|_{t=\tau} = -G[\eta_*] u(\tau)+F(\tau).
\eq
Then we define
\begin{align*}
    \mathcal{G}(u)(t) = j(t) + \int_0^t k_\tau (t) d\tau,
\end{align*}
so that $g^\eps$ solves \eqref{modified} iff $g^\eps$ is a fixed point of $\mathcal{G}$. We note that $\int_{\T^d}\cG(u)(t)dx=0$ since $G[\eta_*]v$ has mean zero for any $v$.

By virtue of Proposition \ref{prop:DN:Sobolevcontinuity}, if $r>1+\frac{d}{2}$, $-\mez\le s\le r-1$, $\eta_*\in H^{s+1}(\T^d)$, and  $v\in H^r$, then $G[\eta_*]v\in H^s$ and
\bq\label{estG:Hs}
\| G[\eta_*]v\|_{H^s}\le C(\| \eta_*\|_{H^r})\| v\|_{H^{s+1}}.
\eq
 Note that $u(\tau)\in \rH^{s+1}$ for a.e. $t\in [0, T]$, whence $-G[\eta_*] u(\tau)+F(\tau)\in \rH^s$ for a.e. $t\in [0, T]$.  In addition, we have $(\p_t -L_\eps)u(x, t)=(\p_t-\Delta)v(x, t)$, $v(x, t)=u(x-\gamma e_1t, t)$. Therefore, the standard theory of the linear heat equation  in Sobolev spaces gives the existence and uniqueness of $j\in Y^s([0, T])$ and $k_\tau\in Y^s([\tau, T])$; moreover, we have 
\begin{align}\label{est:j}
    &\|j\|_{Y^s([0, T])} \leq M\|g_0\|_{H^s} \\ \nonumber
    &\|k_\tau(t)\|_{Y^s([\tau, T])} \leq M\|-G[\eta_*] u(\tau)+F(\tau)\|_{H^s}\\ 
    &\hspace{1in}\le C(\|\eta_*\|_{H^r})\| u(\tau)\|_{H^{s+1}}+M\| F(\tau)\|_{H^s}\quad\text{a.e. } \tau\in [0, T], \label{est:k}
\end{align}
where $M=M(s, d)>0$. 

Next, we bound $\int_0^t k_\tau(t) d\tau$ in $Y^s([0, T])$.  For any $t \in [0,T]$, using \eqref{est:k} gives
    \bq\label{estk:Y:1}
    \begin{aligned}
        \left\|\int_0^t k_\tau(t) d\tau\right\|_{H^s_x} &\leq \int_0^t \|k_\tau(t)\|_{H^s_x} d\tau \\
        &\leq \int_0^tC(\|\eta_*\|_{H^r})\| u(\tau)\|_{H^{s+1}}+M\| F(\tau)\|_{H^s}d\tau \\
        &\leq C(\|\eta_*\|_{H^r}) \sqrt{T} \|u\|_{L^2([0, T];  H^{s+1})}+M\sqrt{T}\| F\|_{L^2([0, T]; H^s)}.
    \end{aligned}
    \eq
    Set $H(t,\tau)=\|k_\tau(t)\|_{H^{s+1}_x}$ for $\tau\in [0, t]$.  Using Minkowski's inequality and \eqref{est:k}, we find
    \bq\label{estk:Y:2}
 \begin{aligned}
       \left\|\int_0^t k_\tau(t) d\tau\right\|_{L^2([0, T];  H^{s+1}_x)} 
       %&\leq \left\| \int_0^t \|k_\tau(t)\|_{H^{s+1}_x} d\tau \right\|_{L^2_t([0, T])} \\
        &\le \| \int_0^t H(t,\tau) d\tau \|_{L^2_t([0, T])} \\
        &= \|\|H(t,\tau)\chi_{[0, t]}(\tau)\|_{L^1_\tau([0,T])} \, \|_{L^2_t([0,T])} \\
        &\leq \|\,\|H(t,\tau)\chi_{[\tau,T]}(t)\|_{L^2_t([0,T])} \, \|_{L^1_\tau([0,T])} \\
        &= \|\,\|k_\tau(t)\|_{L^2_t([\tau,T]; H^{s+1}_x)} \, \|_{L^1_\tau([0,T])} \\
        &\leq \|\, \|k_\tau(t)\|_{Y^s([0, T])_{x, t}} \, \|_{L^1_\tau([0,T])} \\
        &\leq C(\|\eta_*\|_{H^r})\|\| u(\tau)\|_{H^{s+1}}\, \|_{L^1_\tau([0,T])}+M\|\| F(\tau)\|_{H^{s+1}}\, \|_{L^1_\tau([0,T])}  \\
        &\leq C(\|\eta_*\|_{H^r}) \sqrt{T} \|u\|_{L^2([0, T]; H^{s+1})}+M\sqrt{T}\| F\|_{L^2([0, T]; H^s)}.
    \end{aligned}
    \eq
It follows from \eqref{estk:Y:1} and \eqref{estk:Y:2} that 
\bq\label{estk:Y}
\left\|\int_0^t k_\tau(t) d\tau\right\|_{Y^s([0, T])_{x, t}}\le  C(\|\eta_*\|_{H^r}) \sqrt{T}(\eps^{-\mez}+1)\|u\|_{Y^s([0, T_0])}+M\sqrt{T}(1+\eps^\mez)\| F\|_{L^2([0, T]; H^s)}.
\eq
Combining \eqref{est:j} and \eqref{estk:Y} yields 
\bq\label{bound:cG}
\| \mathcal{G}(u)\|_{Y^s([0, T])}\le M\| g_0\|_{H^s}+C(\|\eta_*\|_{H^r}) \sqrt{T}(\eps^{-\mez}+1) \|u\|_{Y^s([0, T])} +2M\sqrt{T}\| F\|_{L^2([0, T]; H^s)}.
\eq
 Set $R=2M\| g_0\|_{H^s}$. Choosing 
\bq\label{cd:T:k}
T_0\le \min\left(\frac{1}{\big(4(\eps^{-\mez}+1)C(\|\eta_*\|_{H^r})\big)^2}, \frac{R^2}{\big(8M\| F\|_{L^2([0, T]; H^s)}\big)^2}\right)
\eq
we deduce from \eqref{bound:cG} that 
 $\cG: \overline{B}_R(0)\to \overline{B}_R(0)$, where $\overline{B}_R(0)$ denotes the closed ball of radius $R$ centered at the origin in $Y^s([0, T_0])$. 

To prove the contraction of $\cG$ we consider any $u_1$, $u_2\in \overline{B}_R(0)\subset Y^s([0, T_0])$. For $\tau\in [0, T_0]$, let  $k_{i, \tau}(x, t)$ be the solution of \eqref{eq:k} with $u$ replaced by $u_i$ and $k_{i, \tau}\vert_{t=\tau}=u_i(\tau)$, $i\in\{1, 2\}$. Since \eqref{eq:k} is linear, $k_\tau:=k_{1, \tau}-k_{2, \tau}$ solves \eqref{eq:k} with $u:=u_1-u_2$, $F=0$, and $k_{\tau}\vert_{t=\tau}=u(\tau)$. Therefore, the estimate \eqref{estk:Y} together with \eqref{cd:T:k} implies
\begin{align*}
\| \cG(u_1)-\cG(u_2)\|_{Y^s([0, T_0])}&= \left\|\int_0^t k_\tau(t) d\tau\right\|_{Y^s([0, T_0])_{x, t}}\\
&\le  C(\|\eta_*\|_{H^r}) \sqrt{T_0} (\eps^{-\mez}+1)\|u_1-u_2\|_{Y^s([0, T_0])}\\
& \le \frac14\|u_1-u_2\|_{Y^s([0, T_0])}.
\end{align*}
We conclude that $\cG$ is a contraction on $\overline{B}_R(0)\subset Y^s([0, T_0])$ provided $T_0$ satisfies \eqref{cd:T:k}. Thus $\cG$ has a unique fixed point $g^\eps$ in $\overline{B}_R(0)$, which solves \eqref{modified} on $[0, T_0]$. Finally, since the restriction \eqref{cd:T:k} depends only on the given $\eps>0$, $\eta_*\in H^r$, $g_0 \in H^s$, and $F\in L^2([0, T]; H^s)$, the solution $g^\eps$ can be extended to a unique solution in $Y^s([0, T])$. The proof of Lemma \ref{lemm:geps} is complete. 
\end{proof}
    \begin{proof}[Proof of Proposition \ref{prop:g}] Let $T>0$ be arbitrary. We assume that $s\in \mathbb{N}$, $\eta_*\in W^{s+2, \infty}(\T^d)\cap H^r(\T^d)$ with  $r>1+\frac{d}{2}$ and $s \le r-1$, and $F\in L^2([0, T]; \rH^{s-\mez}(\T^d))$. We first approximate $F$ by $F^\eps\in L^2([0, T]; \rH^s(\T^d))$, $\eps \in (0, 1)$. Then, Proposition \ref{prop:g} implies that for each $\eps \in (0, 1)$, the problem 
    \bq\label{eq:geps}
      \partial_t g^\eps= L_\eps g^\eps - G[\eta_*]g^\eps+F^\eps, \quad g^\eps|_{t=0} = g_0
    \eq
     has a unique solution $g^\eps\in Y^s([0, T])$.  The proof proceeds in two steps. 
    
    {\it  Step 1.} We prove uniform-in-$\eps$ estimates for $g^\eps$ in $X^s_T\supset Y^s([0, T])$. 
    
     By virtue of the coercive estimate \eqref{coercive} for $G[\eta_*]$, we have 
     \bq\label{coercive:2}
     (G[\eta_*]v, v)_{L^2, L^2}\ge c_0\| v\|_{H^\mez}^2,\quad c_0 = c_0(\|\eta_*\|_{W^{1,\infty}}),
     \eq
     provided $v$ has mean zero. Thus an $L^2$ estimate for \eqref{modified} yields
     \bq\label{est:g:L2}
    \begin{aligned}
        \frac12 \frac{d}{dt} \|g^\eps\|^2_{L^2} &= (\partial_t g^\eps, g^\eps)_{L^2, L^2} 
        = (\gamma \partial_1 g^\eps+\eps \Delta g^\eps - G[\eta_*]g^\eps+F^\eps, g^\eps)_{L^2, L^2}\\
        &=-\eps \| \na g^\eps\|_{L^2}^2 -(G[\eta_*] g^\eps, g^\eps)_{L^2, L^2} +(F^\eps, g^\eps)_{L^2, L^2}\\
        &\leq -\eps \| \na g^\eps\|_{L^2}^2-c_0\|g^\eps\|^2_{H^\frac12}+\| F^\eps\|_{H^{-\mez}}\| g^\eps \|_{H^\mez}\\
        & \leq  -\eps \| \na g^\eps\|_{L^2}^2-\frac{c_0}{2}\|g^\eps\|^2_{L^2}+\frac{1}{2c_0}\| F^\eps\|_{H^{-\mez}}^2,
    \end{aligned}
    \eq
    where we have used the fact that $g^\eps$ has mean zero.
    
    Let $\alpha\in \mathbb{N}^d$ be an arbitrary multiindex of order $s$, $|\alpha|=s$. Applying $\p^\alpha$ to \eqref{modified}, then multiplying the resulting equation by $\p^\alpha g^\eps$, we obtain
    \bq\label{Hs:est:g:1}
    \begin{aligned}
        \frac12 \frac{d}{dt} \|\partial^\alpha g^\eps\|^2_{L^2} & = (\gamma \partial_1 \partial^\alpha g^\eps+\eps\Delta g^\eps - \partial^\alpha G[\eta_*] g^\eps+\p^\alpha F^\eps, \partial^\alpha g^\eps)_{L^2, L^2}\\
        %&=-\eps\| \na \p^\alpha g^\eps\|_{L^2}^2-(\partial^\alpha G[\eta_*] g^\eps, \partial^\alpha g^\eps)_{L^2, L^2}+(\partial^\alpha F^\eps, \partial^\alpha g^\eps)_{L^2, L^2}\\
        &=-\eps\| \na \p^\alpha g^\eps\|_{L^2}^2 -(G[\eta_*] \partial^\alpha g^\eps, \partial^\alpha g^\eps)_{L^2, L^2} + \left([G[\eta_*], \partial^\alpha] g^\eps, \partial^\alpha g^\eps\right)_{L^2, L^2}\\
        &\qquad+(\partial^\alpha F^\eps, \partial^\alpha g^\eps)_{L^2, L^2}.
        \end{aligned}
        \eq
        Clearly 
        \bq\label{Cauchy:Fg}
        |(\partial^\alpha F^\eps, \partial^\alpha g^\eps)_{L^2, L^2}|\le \| F^\eps\|_{H^{s-\mez}}\| g^\eps\|_{H^{s+\mez}}\le \frac{c_0}{4}\| g^\eps\|_{H^{s+\mez}}+\frac{1}{c_0}\| F^\eps\|^2_{H^{s-\mez}}.
        \eq
        Now we apply the commutator estimate \eqref{commutator:est} with $\sigma=\mez$ to have
     \[
        \begin{aligned}
        \left|\left([G[\eta_*], \partial^\alpha] g^\eps, \partial^\alpha g^\eps\right)_{L^2, L^2}\right|&\le \|[G[\eta_*], \partial^\alpha] g^\eps\|_{H^\frac12} \|\partial^\alpha g^\eps\|_{H^{-\frac12}}\\
        &\le C(\|\eta_*\|_{W^{s+2,\infty}})\| g^\eps\|_{H^{s+\mez}}\| g^\eps\|_{H^{s-\mez}}
        \end{aligned}
        \]
        By interpolating $\| g^\eps\|_{H^{s-\mez}}$ between $\| g^\eps\|_{H^{s+\mez}}$ and $\| g^\eps\|_{L^2}$, and applying Young's inequality, we deduce 
 \bq\label{Hs:est:g:2}
  \left|\left([G[\eta_*], \partial^\alpha] g^\eps, \partial^\alpha g^\eps\right)_{L^2, L^2}\right|\le \frac{c_0}{4}\| g^\eps\|_{H^{s+\mez}}^2+ C(\|\eta_*\|_{W^{s+2,\infty}})\| g^\eps\|_{L^2}^2
 \eq    
        Combining \eqref{Hs:est:g:1}, \eqref{Cauchy:Fg}, \eqref{Hs:est:g:2}, and the coercive estimate \eqref{coercive:2}, we obtain 
        \bq\label{eest:geps}
        \begin{aligned}
        \frac12 \frac{d}{dt} \|\partial^\alpha g^\eps\|^2_{L^2}
        &\leq -\eps\| \na \p^\alpha g^\eps\|_{L^2}^2-\frac{c_0}{2} \|g^\eps\|^2_{H^{s+\frac12}} + c_1 \|g^\eps\|^2_{L^2}+c_1\| F^\eps\|_{H^{s-\mez}}^2.
    \end{aligned}
    \eq
    where $c_1 = c_1(\|\eta_*\|_{W^{s+2,\infty}})$. 
    Summing the preceding inequality over all $|\alpha| = s$ yields
    \bq\label{est:g:Hs}
        \frac12 \frac{d}{dt} \sum_{|\alpha| = s}\|\partial^\alpha g^\eps\|^2_{L^2} \leq -M\eps\| \na g^\eps\|_{H^s}^2- c'_0\|g^\eps\|^2_{H^{s+\frac12}} + C_1 \|g^\eps\|^2_{L^2}+C_1\| F^\eps\|^2_{H^{s-\mez}},
    \eq
    where $M=M(d, s)$, $c'_0 = c'_0(\|\eta_*\|_{W^{1,\infty}}, d, s)$, and $C_1= C_1(\|\eta_*\|_{W^{s+2, \infty}}, d, s)$. If we choose $A$ such that $A\frac{c_0}{2} > C_1$, then it follows from \eqref{est:g:L2} and \eqref{est:g:Hs} that
\bq\label{energy:geps}
        E(t) := \frac12 A \|g^\eps\|^2_{L^2} + \frac12  \sum_{|\alpha| = s}\|\partial^\alpha g^\eps\|^2_{L^2}
    \eq
    %is comparable with $\|g^\eps\|^2_{H^s}$, say
   % \begin{align*}
    %    a_0 \|g^\eps\|^2_{H^s} \leq E(t) \leq a_1 \|g^\eps\|^2_{H^s},
    %\end{align*}
    %
    satisfies 
    \bq\label{eest:geps}
        E'(t) \leq - c'_0\|g^\eps\|^2_{H^{s+\frac12}}+C_1\| F^\eps\|^2_{H^{s-\mez}} %\leq -C_0\|g^\eps\|^2_{H^s}+C_1\| F^\eps\|^2_{H^{s-\mez}}
        % \leq -\frac{C_0}{a_1} E(t)+C_1\| F^\eps\|^2_{H^{s-\mez}}.
    \eq 
%    It follows that $E(t)$ decays exponentially in time, 
%    \[
%    E(t)\le E(0)e^{-C_0/a_0t},
%    \]
%    and hence 
% \[
%    \int_0^T \|g^\eps\|^2_{H^{s+\frac12}} \leq \frac{1}{C_0} \int_0^T -E'(t) \le  \frac{E(0)}{C_0}. 
%    \]
    Since $E(t)$ and $\| g^\eps(t)\|_{H^s}^2$ are comparable, integrating the preceding differential inequality yields
    \bq\label{uniest:geps}
    \|g^\eps\|_{X^s_T} \leq C(\|\eta_*\|_{W^{s+2, \infty}})\left(\|g_0\|_{H^s}+\| F^\eps\|_{L^2([0, T]; H^{s-\mez})}\right).
    \eq
    We note that $\| F^\eps\|_{L^2([0, T]; H^{s-\mez})}$ is bounded uniformly in $\eps$.
    
%    \begin{align*}
 %       \|g\|^2_{H^s} &\leq \frac{1}{a_0}E(t) %\leq \frac{1}{a_0} E(0) \leq \frac{a_1}{a_0} \|g(0)\|^2_{H^s} = \frac{a_1}{a_0} \|f_0\|^2_{H^s} \\
  %      \int_0^T \|g\|^2_{H^{s+\frac12}} &\leq %\frac{1}{C_0} \int_0^T -E'(t) = \frac{E(0)}{C_0} \leq \frac{a_1}{C_0}\|f_0\|^2_{H^s},
  %  \end{align*}
%    which allows us to conclude that $\|g\|_{X^s_T} \leq C(\|\eta_*\|_{W^{s+2, \infty}}) \|f_0\|_{H^s}$
{\it Step 2.} We prove contraction estimates for $g^\eps$ in $X^s_T$. Consider $0<\eps'<\eps<1$ and set $g_\sharp=g^\eps-g^{\eps'}$. Then $g_\sharp$ satisfies 
\bq\label{eq:gsharp}
\begin{cases}
\p_tg_\sharp=L_\eps g_\sharp-G[\eta_*]g_\sharp +F_\sharp+\wt{F},\quad F_\sharp:=F^\eps-F^{\eps'},~\wt{F}:=(\eps-\eps')\Delta g^{\eps'},\\
g_\sharp\vert_{t=0}=0,
\end{cases}
\eq
Although \eqref{eq:gsharp} is of the same form as \eqref{eq:g:2}, we cannot directly apply the results in Step 1 because $\wt{F}$ only belongs to $L^2([0, T]; H^{s-1})$. We shall modify  the energy estimates in Step 1 to handle the less regular  forcing term $\wt{F}$.  The idea is to use the dissipation term $\eps \Delta$ instead of $-G[\eta_*]$. By integration by parts and Young's inequality, we have
\begin{align*}
&|(g_\sharp,  \wt{F})_{L^2, L^2}|=|(\na g_\sharp,  (\eps-\eps')\na g^{\eps'})_{L^2, L^2}|\le \frac{\eps}{2}\|\na g_\sharp\|_{L^2}^2+ \frac{(\eps-\eps')^2}{2\eps }\|\na g^{\eps'}\|_{L^2}^2,\\
&\sum_{|\alpha|=s}\left|(\p^\alpha g_\sharp, \p^\alpha \wt{F})_{L^2, L^2}\right|=\sum_{|\alpha|=s}\left|(\p^\alpha \na g_\sharp,   (\eps-\eps')\na \p^\alpha g^{\eps'})_{L^2, L^2}\right|\\
&\hspace{1.5in}\le \frac{M}{2}\eps\|\na g_\sharp\|_{H^s}^2+ M'\frac{(\eps-\eps')^2}{\eps}\| g^{\eps'}\|_{H^{s+1}}^2,
\end{align*}
where $M'=M'(s, d)$. Inserting the preceding estimates in \eqref{est:g:L2} and \eqref{est:g:Hs}, we find that  the energy 
\begin{align*}
        E_\sharp(t) = \frac12 A \|g_\sharp|^2_{L^2} + \frac12  \sum_{|\alpha| = s}\|\partial^\alpha g_\sharp\|^2_{L^2}
    \end{align*}
   satisfies 
   \bq\label{eineq:gsharp}
     E_\sharp'(t) \leq - c'_0\|g_\sharp\|^2_{H^{s+\frac12}}+C_1\| F_\sharp\|^2_{H^{s-\mez}} +M''\frac{(\eps-\eps')^2}{\eps}\| g^{\eps'}\|_{H^{s+1}}^2
     \eq
     provided $A\frac{c_0}{2}>C_1$ as before. Integrating \eqref{eineq:gsharp} and invoking the uniform bound \eqref{uniest:geps}, we deduce 
     \begin{align*}
     \| g_\sharp\|_{X^s_T}^2&\le C(\|\eta_*\|_{W^{s+2, \infty}})\left\{\| F_\sharp\|_{L^2([0, T]; H^{s-\mez})}^2+\frac{(\eps-\eps')^2}{\eps}\| g^{\eps'}\|_{L^2([0, T]; H^{s+1})}^2\right\}\\
     &\le C(\|\eta_*\|_{W^{s+2, \infty}})\left\{\| F_\sharp\|_{L^2([0, T]; H^{s-\mez})}^2+\eps\left(\|g_0\|_{H^s}^2+\| F^{\eps'}\|^2_{L^2([0, T]; H^{s-\mez})}\right)\right\}.
     \end{align*}
     Consequently as $0<\eps'<\eps\to 0$, we have $\| g_\sharp\|_{X^s_T}\to 0$. Therefore, $g_\eps$ converges to some $g$ in $X^s_T$, and $g$ satisfies the bound \eqref{est:g:Xs} upon letting $\eps\to 0$ in \eqref{uniest:geps}.
     
   The convergence $g^\eps\to g$ in $C([0, T]; H^s)$ implies $g\vert_{t=0}=g_0$. On the other hand, using the convergence  $g^\eps\to g$ in $L^2([0, T]; H^{s+1})$ and the linear estimate \eqref{estG:Hs}, we deduce that 
   \begin{align*}
  & G[\eta_*]g^\eps\to G[\eta_*]g\quad\text{in } L^2([0, T]; H^s),\quad \gamma \p_1g^\eps\to \gamma \p_1g\quad\text{in } L^2([0, T]; H^s),\\
   & \eps \Delta g^\eps\to 0\quad\text{in } L^2([0, T]; H^{s-1}),
    \end{align*}
     and $\p_t g^\eps \to \p_t g$ in the distributional sense. Thus letting $\eps\to 0$ in \eqref{modified} we obtain that $g$ is a solution of \eqref{eq:g:2}. Since \eqref{eq:g:2} is linear, the uniqueness of $g$ is a direct consequence of the estimate \eqref{est:g:Xs}.
\end{proof}
\begin{rema}
(i) With a variant of Theorem \ref{theo:commutator} for the whole space $\Rr^d$ in place of $\T^d$, Proposition \ref{prop:g} also holds in $\Rr^d$. 

(ii) The differential energy inequality \eqref{eest:geps} will be used to deduce the asymptotic stability in Theorem \ref{theo:stability}.
\end{rema}
\section{Asymptotic stability of large traveling waves}\label{sec:stability}
Let $\eta_*$ be a traveling wave with speed $\gamma$, i.e. $(\eta_*, \gamma)$ solves \eqref{eq:tw}. Our goal is to prove that $\eta_*$ is stable in Sobolev spaces provided it is close enough to $-\phi$. To that end, we let $f$ denote the perturbation 
    \[
        f(x,t) = \eta(x,t)-\eta_*(x),
    \]
    where $\eta$ solves the dynamic problem \eqref{eq:eta}, i.e.
    \begin{equation} \label{eq:eta:2}
            \partial_t\eta - \gamma \partial_1  \eta = - G[\eta](\eta + \phi).
    \end{equation}
    Then $f$ satisfies 
    \begin{equation}\label{fdt}
    \begin{aligned}
   \partial_t f &= \gamma\partial_1 f -G[\eta](\eta+\phi) + G[\eta_*](\eta_*+\phi)\\
   &= \gamma\partial_1 f -G[f+\eta_*](f+\eta_*+\phi) + G[\eta_*](\eta_*+\phi)\\
   &= \gamma\partial_1 f -G[\eta_*]f +\left\{G[\eta_*]f-G[f+\eta_*]f\right\}+ \left\{ G[\eta_*](\eta_*+\phi)-G[f+\eta_*](\eta_*+\phi)\right\}.
    \end{aligned}
    \end{equation}
%    We first prove that \eqref{fdt} has a unique global solution if the initial perturbation $f_0$ is sufficiently small. This will be achieved in the space 
Our main result of this section is the following. 
\begin{theo}[{\bf Asymptotic Stability}]\label{theo:stability}
    Let $\Nn\ni s > 1 + \frac{d}{2}$,  and assume that $\phi \in H^{s+\frac12}(\T^d)$ and $\eta_* \in W^{s+2, \infty}(\T^d)$.  There exists a nonincreasing function (in each argument) $\om: \Rr_+\times\Rr_+\to \Rr_+ $ and a function $C_* : \Rr_+\times \Rr_+\to \Rr_+$ depending only on $(s, d, b)$ such that if 
    \[
\|\eta_*+\phi\|_{H^{s+\frac12}} < \om(\|\eta_*\|_{W^{s+2,\infty}}, \|\phi\|_{H^s})
    \]
    then  the following holds. For any number  $0<\delta < \om(\|\eta_*\|_{W^{s+2,\infty}}, \|\phi\|_{H^s})$, if $f_0\in \rH^s(\T^d)$ and 
    \[
    \|f_0\|_{H^s} < \frac{\delta}{C_*(\|\eta_*\|_{W^{s+2,\infty}}, \|\phi\|_{H^s})}
    \]
    then  
    the problem
    \bq\label{perturb}
    \begin{cases}
        \partial_t f = \gamma\partial_1 f -G[\eta_*]f +\left\{G[\eta_*]f-G[f+\eta_*]f\right\}+ \left\{ G[\eta_*](\eta_*+\phi)-G[f+\eta_*](\eta_*+\phi)\right\}, \\
        f\vert_{t=0} = f_0
    \end{cases}
    \eq
    has a unique solution $f$ in $X^s_T$  for all $T > 0$. The Banach space $X^s_T$ is defined as in \eqref{def:Xs}. Moreover, there exist positive constants  $C = C(\|\eta_*\|_{W^{s+2,\infty}}, \|\phi\|_{H^s})$, $C_0=C_0(\| \eta_*\|_{W^{s+2, \infty}})$ and $c_0 = c_0(\|\eta_*\|_{W^{s+2,\infty}})$ such that
\begin{align}\label{globalbound:f}        &\|f\|_{X^s_T}\le C\| f_0\|_{H^s},\quad T > 0,\\ \label{decay:est}
&\| f(t)\|_{H^s}\le C_0\| f_0\|_{H^s}e^{-c_0t},\quad t>0.
    \end{align}
\end{theo}
\begin{proof}
The proof proceeds in two steps. 

{\it Step 1.} Global existence of $f$. We first  rewrite \eqref{perturb} in the abstract form
\bq\label{eq:f:2}
    \partial_t f = \cL f + N(f),\quad f |_{t=0} = f_0,
\eq
where  $\cL$ is the linear operator \eqref{def:cL} and $N$ is the nonlinear operator
\bq\label{def:N(f)}
    N(f) = \left\{G[\eta_*]f-G[f+\eta_*]f\right\}+ \left\{ G[\eta_*](\eta_*+\phi)-G[f+\eta_*](\eta_*+\phi)\right\}.
\eq
 We will conduct a contraction mapping argument that results in the existence of a solution $f$ of \eqref{eq:f:2} in $X^s_T$ for all $T>0$. To that end, we fix $T>0$, $f_0\in \rH^s(\T^d)$, and $f\in X^s_T$, where $\| f\|_{C(\T^d\times [0, T])}$ is sufficiently small so that $f+\eta_*>-b$ in the finite depth case. Let $g$ and $h$ be the solutions of the problems
\bq\label{eq:g}
    \partial_t g = \cL g,\quad g\vert_{t=0} = f_0\in \rH^s(\T^d)
\eq
and
\bq\label{eq:h}
    \partial_t h = \cL h + N(f),\quad h\vert_{t=0} = 0.
\eq
%The existence and uniqueness of $g$ and $h$ are proven in Appendix \ref{appendix}. Since $Lk$ and $N(k)$ have spatial mean zero, so do $g$ and $h$.  
Then we define the map 
\[
   X^s_T\ni f\mapsto  \mathcal{F}(f):= g + h,
\]
so that
\[
    \partial_t\mathcal{F}(f) = \cL\mathcal{F}(f) + N(f), \quad\mathcal{F}(f)|_{t=0} = f_0.
\]
Therefore, $f$ is a solution of \eqref{perturb} iff $f$ is a fixed point of $\cF$. Let us show that $\cF$ is well-defined. Since we assume $\eta_*\in W^{s+2, \infty}(\T^d)$ with $s>1+\frac{d}{2}$, we have $\eta_*\in H^{s+2}(\T^d)$ with $s+2>3+\frac{d}{2}$, so $\eta_*$ satisfies the assumptions of Proposition \ref{prop:g}. Therefore, by virtue of Proposition \ref{prop:g}, $g$ is well-defined in $X^s_T$ and so is $h$ provided $N(f)\in L^2([0, T]; \rH^{s-\mez})$. To verify this  we apply the contraction estimate \eqref{Sobolevcontraction} with $(s_0, s)$ replaced by $(s, s+\mez)$ to have
\bq\label{est:N(f):0}
\begin{aligned}
    \|N(f)\|_{H^{s-\frac12}} &= \left\|\left\{G[\eta_*]f-G[f+\eta_*]f\right\}+ \left\{ G[\eta_*](\eta_*+\phi)-G[f+\eta_*](\eta_*+\phi)\right\}\right\|_{H^{s-\frac12}} \\
    &\leq \wt C(\|\eta_*\|_{H^{s}}, \|f+\eta_*\|_{H^{s}})\left\{ \|f\|_{H^s}\|f\|_{H^{s+\frac12}}+\|f\|_{H^s}^2\big(\| \eta_*\|_{H^{s+\mez}}+\| f+\eta_*\|_{H^{s+\mez}}\big)\right\}\\
    &\qquad+ \wt C(\|\eta_*\|_{H^{s}}, \|f+\eta_*\|_{H^{s}})\left\{\|f\|_{H^{s+\mez}}\|\eta_*+\phi\|_{H^{s+\frac12}}\right.\\    &\hspace{2in}\left.+\| f\|_{H^s}\| \eta_*+\phi\|_{H^s}\big(\| \eta_*\|_{H^{s+\mez}}+\| f+\eta_*\|_{H^{s+\mez}}\big)\right\}\\
    &\le \wt{C}_1(\| \eta_*\|_{H^{s+\mez}}, \|\phi\|_{H^s}, \| f\|_{H^s})\left\{\|f\|_{H^s}\|f\|_{H^{s+\frac12}}+ \|f\|_{H^{s+\mez}}\|\eta_*+\phi\|_{H^{s+\frac12}}\right\}.
  %  &\le c_1 \|f\|_{H^s}\|f\|_{H^{s+\frac12}} + c_1 \|f\|_{H^{s+\frac12}}\|\eta_*+\phi\|_{H^{s+\frac12}}.
\end{aligned}
\eq
Note that $s>1+\frac{d}{2}$ suffices.  It follows that
\bq\label{est:N(f)}
\begin{aligned}
\| N(f)\|_{L^2([0, T]; H^{s-\mez})}&\le \wt{C}_1(\| \eta_*\|_{H^{s+\mez}}, \|\phi\|_{H^s}, \| f\|_{L^\infty([0, T], H^s)})\left\{\| f\|_{L^\infty([0, T]; H^s)}\|f\|_{L^2([0, T]; H^{s+\frac12})}\right.\\
&\hspace{1in}\left. +\|f\|_{L^2([0, T]; H^{s+\frac12})}\|\eta_*+\phi\|_{H^{s+\frac12}}\right\}\\
&\le \wt{C}_1(\| \eta_*\|_{H^{s+\mez}}, \|\phi\|_{H^s}, \| f\|_{X^s_T})\left\{\|f\|_{X^s_T}^2+\|\eta_*+\phi\|_{H^{s+\frac12}}\| f\|_{X^s_T}\right\}<\infty.
\end{aligned}
\eq
In addition, $\int_{\T^d}N(f)dx=0$ because $\int_{\T^d}G[\eta]vdx=0$ for all $v$. Thus $N(f)\in L^2([0, T]; \rH^{s-\mez})$ as claimed, whence $\cF$ is well-defined. Moreover, the estimate \eqref{est:g:Xs} gives 
\begin{align}\label{est:g:10}
&\| g\|_{X^s_T}\le C(\| \eta_*\|_{W^{s+2, \infty}})\| f_0\|_{H^s},\\ \label{est:h:10}
&\| h\|_{X^s_T}\le C(\| \eta_*\|_{W^{s+2, \infty}})\| N(f)\|_{L^2([0, T]; H^{s-\mez})}.
\end{align}
%Now suppose $\|f\|_{L^\infty([0,T]; H^s)} \leq 1$. Then for $c_1=\wt C(\| \eta_*\|_{H^s}, \| \eta_*\|_{H^s}+1)$, we have 
%\bq\label{est:N(f)}
%\|N(f)\|_{H^{s-\frac12}}\le c_1  \|f\|_{H^s}\|f\|_{H^{s+\frac12}} + c_1 \|f\|_{H^{s+\frac12}}\|\eta_*+\phi\|_{H^{s+\frac12}}.
%\eq
%\begin{align*}
 %   \int_0^T \|N(f)\|^2_{H^{s-\frac12}} dt &\leq \int_0^T c'_1 \|f\|^2_{X^s_T}\|f\|^2_{H^{s+\frac12}} + c'_1 \|\eta_* + \phi\|^2_{H^{s+\frac12}} \|f\|^2_{H^{s+\frac12}} dt \\
 %   &\leq c'_1 \|f\|^4_{X^s_T} + c'_1 \|\eta_* +\phi\|^2_{H^{s+\frac12}} \|f\|^2_{X^s_T}
%\end{align*}
%Inserting this estimate into \eqref{est:h}, we find that $h$ satisfies
%\bq\label{h:bound}
%\begin{aligned}
  %  \|h\|_{X^s_T} \leq C_2(\|\eta_*\|_{W^{s+2,\infty}}) (\|f\|^2_{X^s_T} + \|f\|_{X^s_T}\|\eta_* + \phi\|_{H^{s+\frac12}}),
%\end{aligned}
%\eq
%where we have used the fact that $\| \eta_*\|_{H^s}\le C(d, s)\| \eta_*\|_{W^{s+2, \infty}}$.
Now we restrict $\mathcal{F}$ to a ball $\overline{B}_\delta(0)\subset X^s_T$ for some $0 < \delta < 1$. Then it follows from the estimates \eqref{est:g:10}, \eqref{est:h:10}, and \eqref{est:N(f)} that for some $C_1=C_1(\| \eta_*\|_{W^{s+2, \infty}}, \|\phi\|_{H^s})$, 
\bq\label{F(f):bound}
\begin{aligned}
    \|\mathcal{F}(f)\|_{X^s_T} &\leq \|g\|_{X^s_T} + \|h\|_{X^s_T} \\
    &\leq C_1 \left\{\|f_0\|_{H^s} +\|f\|^2_{X^s_T} + \|\eta_* + \phi\|_{H^{s+\frac12}}\|f\|_{X^s_T}\right\} \\
    &\leq C_1 \left\{\|f_0\|_{H^s} + \delta^2 + \|\eta_* + \phi\|_{H^{s+\frac12}} \delta\right\}.
\end{aligned}
\eq
We  assume that 
\bq\label{cd:etaphi:1}
\|\eta_* + \phi\|_{H^{s+\frac12}} < \frac{1}{3C_1}
\eq
and choose $\delta$ and $f_0$ satisfying
\bq\label{cd:deltaf0:1}
    \delta < \frac{1}{3C_1}, \quad \|f_0\|_{H^s} < \frac{\delta}{3C_1},
\eq
so that $\mathcal{F}$ maps $\overline{B}_\delta(0)\subset X^s_T$ into itself.

Next, we show the contraction of $\cF$. Suppose that $f_j\in \overline{B}_\delta(0)\subset X^s_T$. We note that $\mathcal{F}(f_1) - \mathcal{F}(f_2)=h_1-h_2$, where $h_j$ solves \eqref{eq:h} with the right-hand side $N(f_j)$. Since $h: = h_1 - h_2$ solves
\[
    \partial_t h = \cL h + N(f_1) - N(f_2),\quad h|_{t=0} = 0,
\]
The estimate \eqref{est:g:Xs} implies
\bq\label{esth:Xs:0}
    \|h\|_{X^s_T} \leq C(\|\eta_*\|_{W^{s+2, \infty}})\|N(f_1)-N(f_2)\|_{L^2([0, T]; H^{s-\frac12})}.
\eq
In order to bound $N(f_1)-N(f_2)$, we first rewrite 
\begin{align*}
N(f_1)-N(f_2)&=\left\{G[\eta_*](f_1-f_2) - G[f_1+\eta_*](f_1-f_2)\right\} -\left\{G[f_1+\eta_*]f_2-G[f_2+\eta_*]f_2\right\}\\
&\qquad-\left\{G[f_1+\eta_*](\eta_*+\phi)-G[f_2+\eta_*](\eta_*+\phi)\right\},
\end{align*}
so that each term on the right-hand side has the form $G[\eta_1]k-G[\eta_2]k$.  Applying the contraction estimate \eqref{Sobolevcontraction} with $(s_0, s)$ replaced by $(s, s+\mez)$, we obtain as in \eqref{est:N(f):0} that
\begin{align*}
    \|N(f_1)-N(f_2)\|_{H^{s-\frac12}} %&= \|\left\{G[\eta_*]f_1-G[f_1+\eta_*]f_1\right\}+ \left\{ G[\eta_*](\eta_*+\phi)-G[f_1+\eta_*](\eta_*+\phi)\right\} \\
    %& \quad -\left\{G[\eta_*]f_2-G[f_2+\eta_*]f_2\right\} - \left\{ G[\eta_*](\eta_*+\phi)-G[f_2+\eta_*](\eta_*+\phi)\right\}\|_{H^{s-\frac12}} \\
   % &= \|G[\eta_*](f_1-f_2) - G[f_1+\eta_*](f_1-f_2) + G[f_2+\eta_*]f_2-G[f_1+\eta_*]f_2\|_{H^{s-\frac12}} \\
   % &\leq C(\|\eta_*\|_{H^{s-\frac12}}, \|f_1+\eta_*\|_{H^{s-\frac12}})\left\{ \|f_1\|_{H^s}\|f_1-f_2\|_{H^{s+\frac12}} + \|f_1\|_{H^{s+\frac12}}\|f_1-f_2\|_{H^s} \right\} \\
   % & \quad + C(\|f_2+\eta_*\|_{H^{s-\frac12}}, \|f_1+\eta_*\|_{H^{s-\frac12}}) \left\{ \|f_2\|_{H^s}\|f_1-f_2\|_{H^{s+\frac12}} + \|f_2\|_{H^{s+\frac12}}\|f_1-f_2\|_{H^s} \right\} \\
    &\leq \wt{C}_2(\|\eta_*\|_{H^{s+\mez}},  \|\phi\|_{H^s})\left\{(\|f_1\|_{H^s}+\|f_2\|_{H^s})\|f_1-f_2\|_{H^{s+\frac12}} \right. \\
    & \qquad \left. + (\|f_1\|_{H^{s+\frac12}}+\|f_2\|_{H^{s+\mez}})\|f_1-f_2\|_{H^s}+ \|f_1-f_2\|_{H^{s+\mez}}\| \eta_*+\phi\|_{H^{s+\mez}}\right\}.
\end{align*}
Taking the $L^2$ norm in time and invoking \eqref{esth:Xs:0}, we obtain for some $C_2=C_2(\|\eta_*\|_{W^{s+2, \infty}}, \|\phi\|_{H^s})$ that 
\bq\label{contraction:N(f)}\begin{aligned}
    \| h\|_{X^s_T} &\leq C_2\left\{2(\|f_1\|_{X^s_T} + \|f_2\|_{X^s_T}) \|f_1-f_2\|_{X^s_T}\right.\\
    &\qquad\left.+\| \eta_*+\phi\|_{H^{s+\mez}} \|f_1-f_2\|_{X^s_T}\right\}\\
    &\leq C_2\left(4\delta\|f_1-f_2\|_{X^s_T}+\| \eta_*+\phi\|_{H^{s+\mez}} \|f_1-f_2\|_{X^s_T}\right).
\end{aligned}
\eq
In view of \eqref{cd:etaphi:1}, \eqref{cd:deltaf0:1}, \eqref{contraction:N(f)}, we impose 
\bq\label{smallnesscds:existence}
\begin{aligned}
    &\|\eta_* + \phi\|_{H^{s+\frac12}} <\min( \frac{1}{3C_1}, \frac{1}{3C_2}),\\
    &\delta < \min(\frac{1}{3C_1}, \frac{1}{12C_2}), \quad \|f_0\|_{H^s} < \frac{\delta}{3C_1},
\end{aligned}
\eq
so that $\mathcal{F}$ is a contraction on $\overline{B}_\delta(0)$. Therefore, $\cF$ has   a unique fixed point $f \in \overline{B}_\delta(0)\subset X^s_T$  which is also the unique solution of  \eqref{perturb} in $\overline{B}_\delta(0)$. Since the smallness conditions in \eqref{smallnesscds:existence} are independent of the time $T$, we in fact obtain a global solution to \eqref{perturb}. 

Using the second inequality in \eqref{F(f):bound}, \eqref{cd:etaphi:1}, and \eqref{cd:deltaf0:1}, we find
\begin{align*}
    \|f\|_{X^s_T} &\leq C_1 \left\{\|f_0\|_{H^s} + \|f\|^2_{X^s_T} + \|f\|_{X^s_T}\|\eta_* + \phi\|_{H^{s+\frac12}}\right\} \\
   % &\leq C_1 \|f_0\|_{H^s} + C_2 \delta \|f\|_{X^s_T} + C_2 \|\eta_*+\phi\|_{H^{s+\frac12}} \|f\|_{X^s_T} \\
    &\leq C_1 \|f_0\|_{H^s} + \frac13 \|f\|_{X^s_T} + \frac13 \|f\|_{X^s_T},
\end{align*}
thereby obtaining the global-in-time bound \eqref{globalbound:f}. 

{\it Step 2.} Exponential decay of $f$. We note that \eqref{eq:f:2} is of the form \eqref{eq:geps} with $\eps=0$ and $F^\eps=N(f)$. Therefore, we can apply the estimate \eqref{eest:geps} to have 
 \bq
 E'(t) \leq -c'_0 \|f\|^2_{H^{s+\frac12}}+C_3\| N(f)\|^2_{H^{s-\mez}},
 \eq
where $c'_0 =c'_0(\|\eta_*\|_{W^{1, \infty}})$, $C_3=C_3(\|\eta_*\|_{W^{s+2, \infty}})$, and  the energy $E(t)$ is defined by \eqref{energy:geps}, i.e.
\[
    E(t) = \frac12 A \|f\|^2_{L^2} + \frac12  \sum_{|\alpha| = s}\|\partial^\alpha f\|^2_{L^2},\quad A=A(\|\eta_*\|_{W^{s+2, \infty}}).
\]
Invoking the estimate \eqref{est:N(f):0} for $N(f)$ and recalling that  $\| f\|_{L^\infty((0, \infty); H^s)}\le \delta<1$, we deduce that
\begin{align*}
    E'(t) \leq -c'_0\|f\|^2_{H^{s+\frac12}} + C_4 \|f\|^2_{H^s}\|f\|^2_{H^{s+\frac12}} + C_4 \|f\|^2_{H^{s+\frac12}}\|\eta_*+\phi\|^2_{H^{s+\frac12}},
\end{align*}
where $C_4=C_4(\|\eta_*\|_{W^{s+2, \infty}}, \| \phi\|_{H^s})$. Therefore, if we assume in addition to \eqref{smallnesscds:existence} that
\bq\label{cd:delta:3}
\|\eta_*+\phi\|^2_{H^{s+\frac12}}<\frac{c'_0}{3C_4},\quad\delta^2<\frac{c'_0}{3C_4},
\eq
then
\[
E'(t)\le -\frac{c'_0}{3}\|f\|^2_{H^{s+\frac12}}\le  -\frac{c'_0}{3}\|f\|^2_{H^{s}}.
\]
Consequently $E(t)$ decays exponentially, and hence so does $\|f(t)\|_{H^s}$. In view of \eqref{smallnesscds:existence} and \eqref{cd:delta:3}, we conclude the proof of Theorem \ref{theo:stability}  with 
\bq\label{def:om}
\om=\min\left(\frac{1}{3C_1}, \frac{1}{12C_2}, \sqrt{\frac{c'_0}{3C_4}}\right),\quad C_*=3C_1=3C_1.
\eq
 We recall that $C_1$, $C_2$ and $C_4$ are functions of $(\| \eta_*\|_{W^{s+2, \infty}}, \|\phi\|_{H^s})$, and $c_0'$ is a function of $\| \eta_*\|_{W^{1, \infty}}$. In the above analysis, $c_0'$ can be replaced by any smaller positive constant. Thus, we can assume without loss of generality that $c_0'$ is a nonincreasing function, so that $c_0'(\| \eta_*\|_{W^{1, \infty}})\ge c_0'(\| \eta_*\|_{W^{s+2, \infty}})$. Consequently, we can replace $\om$ in \eqref{def:om} with a smaller constant that depends only on $(\| \eta\|_{W^{s+2, \infty}}, \|\phi\|_{H^s})$. 
%We can replace $\om$ with the smaller function  
%\[
%\om_1(z_1, z_2, z_3):=\inf_{z'_j\in (0, z_j]}\om(z'_1, z_2', z_3')
%\]
%to ensure that $\omega$ is nonincreasing in each argument. 
\end{proof}

\vspace{.1in}
{\noindent{\bf{Acknowledgment.}}  The authors  were partially supported by NSF grant DMS-2205734. 

}

\end{document}